\documentclass[12pt]{amsart}
\usepackage{amsmath, amsthm, mathabx, amsfonts, amssymb, xspace, overpic, color}
\usepackage{graphicx,psfrag,dsfont,amscd,mathrsfs}
\usepackage[all]{xy}
\usepackage{xr}
\usepackage[dvipsnames]{xcolor}
\usepackage{pstricks,pstricks-add}
\pdfoutput=1

\usepackage[hyperfootnotes=false]{hyperref}
\hypersetup{
  colorlinks,
  citecolor=Purple,
  linkcolor=Black,
  urlcolor=Purple}

\usepackage[alphabetic,initials]{amsrefs}
\theoremstyle{plain}
\newtheorem{thm}{Theorem}[section]
\newtheorem{prop}[thm]{Proposition}
\newtheorem{lemma}[thm]{Lemma}

\newtheorem*{thrm}{Main Theorem}
\newtheorem*{claim}{Claim}

\numberwithin{equation}{section}
\numberwithin{thm}{section}

\theoremstyle{definition}
\newtheorem{definition}[thm]{Definition}

\newtheorem*{notation}{\normalfont \textit{{Notation}}}

\theoremstyle{remark}

\newtheorem{example}[thm]{Example}

	  \newcommand{\C}{\mathbb{C}}
	  \newcommand{\R}{\mathbb{R}}
	  \newcommand{\D}{\mathbb{D}}
	  \newcommand{\RP}{\mathbb{R}P}
	  \newcommand{\SH}{\mathcal{H}}
	  \newcommand{\SM}{\mathcal{M}}

	  \newcommand{\ct}{T^* S^}

	  \newcommand{\sm}{\smallskip}
	  \newcommand{\op}{\operatorname}

      \makeatletter      
      \newcommand*{\defeq}{\mathrel{\rlap{%
                     \raisebox{0.3ex}{$\m@th\cdot$}}%
                     \raisebox{-0.3ex}{$\m@th\cdot$}}%
                     =}
      \makeatother
      
      \subjclass[2010]{Primary 32Q65, 57R17; Secondary 53D35}

   \author{Bahar Acu}
   
      \address{Department of Mathematics, Northwestern University, Evanston, IL 60208}
   \email{baharacu@northwestern.edu} 
   \urladdr{http://www.math.northwestern.edu/~baharacu/}
   
   \title[The Weinstein Conjecture for Iterated Planar Contact Structures]{The Weinstein Conjecture for Iterated Planar Contact Structures}
   
   \thanks{The author acknowledges support from the U.S. National Science Foundation. She was partially supported by the NSF grant of Ko Honda (DMS-1406564).}
        
\begin{document}
\begin{abstract}
 In this paper, we introduce the notions of an iterated planar open book decomposition and an iterated planar Lefschetz fibration and prove the Weinstein conjecture for contact manifolds supporting an open book that has iterated planar pages. For $n\geq 1$, we show that a $(2n+1)$-dimensional contact manifold $M$ supporting an iterated planar open book decomposition satisfies the Weinstein conjecture.
  \end{abstract}
  \maketitle
  
\section{Introduction}\label{intro}
The Weinstein conjecture \cite{We} states that, on a compact contact manifold, any Reeb vector field carries at least one closed orbit. The conjecture was formulated by Alan Weinstein in 1978. Several methods from symplectic geometry have been used to prove the conjecture in many cases. However, it is still not known if those methods can be used to prove the conjecture in all cases.\sm

There are many results that move us towards understanding the Weinstein conjecture in both 3 and higher dimensions such as results of Etnyre \cite{Etn}, Floer, Hofer and Viterbo \cite{FHV}, \cite{H}, \cite{HV}, \cite{Vit}, and also of Taubes \cite{T}. The conjecture was proven for all closed 3-dimensional contact manifolds by Taubes \cite{T} by using Seiberg-Witten theory. However, it is still open in higher dimensions. Yet there are partial results analyzing the conjecture in higher dimensions. In \cite{Vit}, the conjecture was proven for all closed contact hypersurfaces in $\R^{2n}$, $n\geq 2$. In \cite{HV} and \cite{FHV}, the result was extended to cotangent bundles by using the presence of holomorphic spheres and also to a large class of aspherical manifolds, respectively.\sm

In this paper we examine the Weinstein conjecture in the case of a special class of higher dimensional contact manifolds. Briefly, a \textit{Weinstein domain} $(W, \omega=d\lambda)$ is a compact exact symplectic manifold with boundary equipped with a Morse function $\phi: W \to \R$ which has $\partial W=\{\phi= 0\}$ as a regular level set and a Liouville vector field which is gradient like for $\phi$. \sm

We say that a $2n$-dimensional Weinstein domain $(W^{2n}, \omega)$ admits an \textit{iterated planar Lefschetz fibration} if there exists a sequence of symplectic Lefschetz fibrations $f_2, \dots, f_{n}$ where $f_{i}: W^{2i} \to \D$, where $\D$ is a 2-disk and $i=2, \dots, n$, such that each regular fiber of $f_{i+1}$ is the total space of $f_i$ and $f_2: W^4 \to \D$ is a Lefschetz fibration whose fibers are planar surfaces, i.e. $f_2$ is a planar Lefschetz fibration. \sm

For $n\geq 2$, $\ct n$ and $A_k$-singularities given by the variety $$\{(z_1, \dots, z_n) \mid z_1^2+\dots+z_{n-1}^2+z_{n}^{k+1}=1\} \subset (\C^n, \omega_{std})$$ for $n\geq 3$ are two significant classes of Weinstein domains admitting an iterated planar Lefschetz fibration. See Section \ref{sec:IPLF} for more details. \sm

Finally, an \textit{iterated planar contact manifold} is a contact manifold supporting an open book whose pages admit an iterated planar Lefschetz fibration.\sm

In this paper, we prove the following theorem.

   \begin{thrm}\label{main_theorem}
Let $(M, \lambda)$ be a $(2n+1)$-dimensional iterated planar contact manifold. For any $\lambda$ on $M$, the associated Reeb vector field on $M$ admits a closed Reeb orbit.
   \end{thrm}
   
Notice that when $n=1$, $M$ is a planar contact manifold and in that case, the Weinstein conjecture is known to be true \cite{ACH}.

\subsection*{Plan of the paper} 
In Section \ref{sec:backg}, we will establish some standard background and notation. In Section \ref{sec:IPLF}, we will introduce two new notions; iterated planar Lefschetz fibration and its associated open book. Here we also remind the reader the necessary background on Lefschetz fibrations and their compatibility with almost complex structures. The remainder of the paper is then devoted to the proof of the Main Theorem which has three main subsections discussing regularity (Section \ref{subsec:regularity}), compactness (Section \ref{subsec:compactify}), dimension (Section \ref{indexformula}) of the moduli space of planar $J$-holomorphic curves in $\R \times M$ around a singular fiber and away from a singular fiber in the Lefschetz fibration setting by using the iterative argument on Lefschetz fibrations. By using these results, we will argue that the symplectization of $M$ is filled by planar $J$-holomorphic curves away from the binding orbits, i.e. there exists a unique curve, up to algebraic count, passing through every point in $\R \times M$. The rest of the proof is then just an application of a neck-stretching argument that helps us observe the desired closed Reeb orbits.\sm

\subsection*{Acknowledgments} The author thanks her advisor Ko Honda whose support and guidance were central throughout the completion of this project and Chris Wendl for very instructive discussions on Lemma \ref{lemma:wendl} and on the earlier versions of the paper. The author would also like to thank University of California, Los Angeles for their hospitality throughout the completion of this project.

\section{Background}\label{sec:backg}

In this section, we will outline some prerequisite material on symplectic field theory which is essential in the proof of the Main Theorem. \textit{Throughout the paper, all manifolds are smooth oriented and dimension means real dimension unless stated otherwise}. 

\subsection{Weinstein domains}

\begin{definition}
A codimension 1 hypersurface $M$ of a symplectic manifold $(W, \omega)$ is called \textit{contact type} if there is a neighborhood $N$ of $M$ such that $\omega=d\lambda$ for some $\lambda \in \Omega^{1}(N)$ and a vector field $Z$ on $W$ determined by $\lambda=\iota_Z \omega$ positively transverse to $M$.
\end{definition}
Observe that the second condition implies that $\lambda|_{M}$ is a contact form.

\begin{definition} A \textit{Liouville domain} is a triple $(W, \omega, Z)$, where  
\begin{enumerate}
\item $\omega$ is a symplectic form on $W$.
\item $(W, \omega=d\lambda)$ is a compact exact $2n$-dimensional symplectic manifold with boundary.
\item $Z$ is a Liouville vector field on $W$, i.e., $\mathcal{L}_Z \omega=\omega$ and points outward at the boundary.
\end{enumerate}
\end{definition}

A \textit{Liouville form} is a contact form $\lambda=\iota_Z \omega$ defining the Liouville vector field $Z$. In this paper, we are primarily interested in the following class of symplectic manifolds whose boundaries are contact hypersurfaces:

\begin{definition}Let $W$ be a compact $2n$-dimensional manifold with boundary.  A {\em{Weinstein domain structure}} on $W$ is a triple $(\omega, Z, \phi)$, where 
\begin{enumerate}
\item $\omega=d\lambda$ is a symplectic form on $W$.
\item $Z$ is a Liouville vector field defined by $\lambda=\iota_{Z} \omega$.
\item $\phi : W \rightarrow \mathbb{R}$ is a generalized Morse function for which $Z$ is gradient-like. 
\item $\phi$ has $\partial W=\{\phi=0\}$ as a regular level set.
\end{enumerate}
\end{definition}	

The quadruple $(W, \omega, Z, \phi)$ is then called a \emph{Weinstein domain}. 

\begin{notation} Throughout, we will abuse the notation and write $(W, \omega)$ instead of $(W, \omega, Z, \phi)$.
\end{notation}

Observe that any Weinstein domain is a Liouville domain. Also, any Weinstein domain $(W, \omega)$ has the contact type boundary $(\partial W, \xi=\op{ker}\lambda)$, where  $\lambda=\iota_Z \omega$. Moreover, a contact manifold $(M, \xi)$ is called \textit{Weinstein fillable} if it is the boundary of a Weinstein domain $(W, \omega)$.

\begin{definition} Let $(M, \xi= \op{ker} \lambda)$ be a contact manifold. The \textit{symplectization} of $M$ is an open symplectic manifold $(\R \times M, d(e^s \lambda)$ where $s$ is the $\R$-coordinate.
\end{definition}

The contact condition on $M$ implies that $d(e^s \lambda)$ is a symplectic form on the symplectization $\R \times M$. Observe that $\partial_s$ is a Liouville vector field transverse to the slices $\{s\}\times M$. Thus, every contact manifold is a contact type hypersurface in its own symplectization $(\R \times M, d(e^s \lambda))$.

\begin{definition}If $(M, \xi)$ is the boundary of a Weinstein domain $(W, \omega=d\lambda)$, then one can smoothly attach a cylindrical end to define
a larger symplectic manifold 
\begin{equation}
(\widehat{W}, \hat{\omega})=(W, \omega) \cup_{\partial W} ([0, \infty) \times M, d(e^s \lambda)) 
\end{equation}
called the \textit{symplectic completion} of $(W, \omega)$.
\end{definition}

\subsection{Almost complex structures} Let $M$ be a $(2n+1)$-dimensional contact manifold and $\lambda$ be a nondegenerate contact form on $M$. Denote by $R_\lambda$ the Reeb vector field for $\lambda$. 

\begin{definition} Let $W$ be a $2n$-dimensional manifold. An \textit{almost complex structure} is an
endomorphism $J: TW \rightarrow TW$ such that $J^2= -1$. $(W, J)$ is then called an \textit{almost complex manifold}. 
\end{definition}

\begin{definition} \label{defn:compatible}
An almost complex structure $J$ on a symplectic manifold $(W, \omega)$ is said to be \textit{compatible} with $\omega$ (or \textit{$\omega$-compatible}) if
\begin{enumerate}
\item $\omega(Ju, Jv)=\omega(u, v)$, for any $u, v \in TW$ and
\item $\omega(v, Jv) > 0$, for any nonzero $v \in TW$.
\end{enumerate}
\end{definition}

It is known that the set of all almost complex structures on a symplectic manifold $(W, \omega)$ is nonempty and contractible. Therefore, the tangent bundle $TW$ can be considered as a complex vector bundle. 

\begin{definition} Let $(F, j)$ be a Riemann surface and $(W, J)$ be an almost complex manifold. A \textit{$J$-holomorphic curve (or pseudoholomorphic curve)} is a smooth map $u : F \rightarrow W$  such that its differential, at every point, satisfies the following complex-linearity condition:
\begin{equation*}
du \circ j = J \circ du
\end{equation*} 
or, equivalently, $\overline{\partial}_J u=0$.
\end{definition}

We will now define a special class of compatible almost complex structures on $(\R \times M, d(e^s \lambda))$. Since $\R \times M$ is noncompact, an almost complex structure $J$ on $\R \times M$ needs to satisfy certain conditions near infinity in order to obtain a well behaved space of $J$-holomorphic curves. Note that the symplectization of $M$ inherits a natural splitting of the tangent bundle $T(\R \times M)=\R\langle \partial_s\rangle \oplus \R\langle R_\lambda\rangle \oplus \xi$, where $\partial_s$ is the unit vector in the $\R$-direction. We will use this information to make the following definition:

\begin{definition}
An almost complex structure $J$ on $\R \times M$ is called \emph{$\lambda$-compatible} if
 \begin{enumerate}
 \item $J$ is $\R$-invariant.
\item $J(\partial_s)=R_\lambda$ and $J(-R_\lambda)=\partial_s$, where $\partial_s$ is the unit vector in the $\R$-direction.
\item $J(\xi)=\xi$.
\item $J|_{\xi}$ is $d\lambda$-compatible.
 \end{enumerate}
\end{definition}

\begin{notation}We will denote by $\mathcal{J}(\R \times M)$ the space of all $\lambda$-compatible almost complex structures on $\R \times M$.
 \end{notation}
 
 \begin{definition}\label{defn:Mcomp} Let $(M, \xi)$ be a contact type hypersurface in a Weinstein domain $(W, \omega=d\lambda)$. An $\omega$-compatible almost complex structure $J$ on $W$ is called \textit{$(M, \lambda)$-compatible} if $J$ restricts to a $\lambda|_{M}$-compatible almost complex structure on a collar neighborhood of $M$.
 \end{definition}
 
 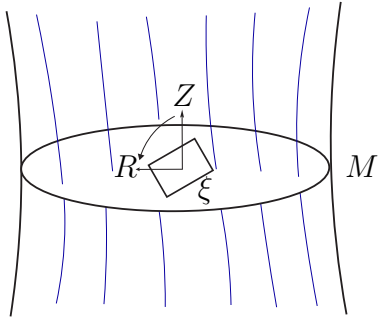
\begin{figure}[h]
\vspace{-.5cm}
\begin{overpic}[scale=.7]{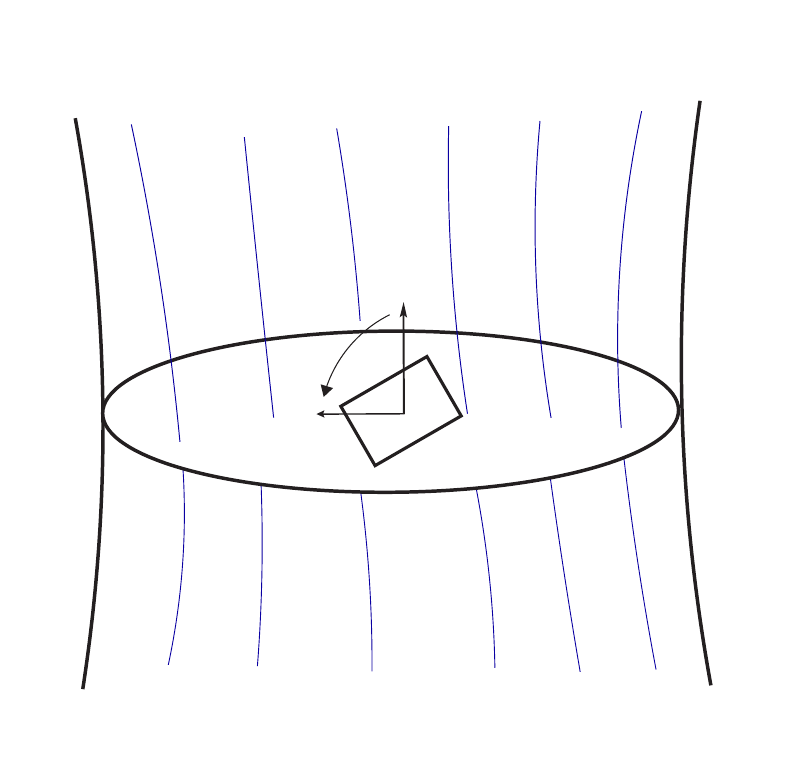}
\put(50,39){$ $}
\put(90, 45){$M$}
\put(55, 40){$\xi$}
\put(35, 45){$R$}
\put(49, 62){$Z$}
\end{overpic}
\vspace{-.5cm}
\caption{Figure depicting the \textit{$(M, \lambda)$-compatible} almost complex structure in a neighborhood of a contact hypersurface $M$. Blue lines represent the flow lines of the Liouville vector field $Z$. Here $R$ is the Reeb vector field for the contact form on $M$.}
 \label{adapted}
\end{figure}

 \begin{definition} A \textit{positive end} (resp., \textit{negative end}) of a $J$-holomorphic curve $u$ is a neighborhood of a puncture on which $u$ is asymptotic to some ${\mathbb R}\times\gamma$ as $s\to\infty$ (resp., $s\to -\infty$), where $\gamma$ is a closed Reeb orbit. The $J$-holomorphic curve $u$ is then called \textit{asymptotically cylindrical}.
\end{definition}

\subsection{Stable Hamiltonian structures}
\begin{definition}A \textit{stable Hamiltonian structure} (SHS) $\mathcal{H}$ on a $(2n+1)$-dimensional manifold $M$ is a pair $(\Omega, \Lambda)$ consisting of a closed 2-form $\Omega$ and 1-form $\Lambda$ defined on $M$ with the following properties:
\begin{enumerate}
\item $\op{ker}\Omega \subset \op{ker}d\Lambda$.
\item $\Lambda \wedge \Omega^{n}>0$.
\end{enumerate}
\end{definition}

Note that the condition $(1)$ can equivalently be written as $d\Lambda=g\omega$ where $g:M \to \R$ is a smooth function. Observe also that the condition $(2)$ is equivalent to $\Omega|_{\xi}$ is nondegenerate, where $\xi \defeq \op{ker}\Lambda$ is a co-oriented hyperplane distribution. In other words, $\Omega|_{\xi}$ is a symplectic vector bundle. Moreover, on a neighborhood $(-\epsilon, \epsilon) \times M$ of $M$, for $\epsilon > 0$ sufficiently small, $d(s\Lambda)+\Omega$ is a symplectic form where $s \in (-\epsilon, \epsilon)$. One can generalize this further to $\R \times M$ by letting the symplectic form on $\R \times M$ to be $\omega_\phi=d(\phi(s)\Lambda)+\Omega$ where $\phi: \R \to (-\epsilon, \epsilon)$ is a strictly increasing function. Observe that the symplectization of $(M, \SH)$ is symplectomorphic to $(\R \times M, \omega_{\phi})$.\sm

\begin{definition} The \textit{Reeb vector field} $R_{\SH}$ of a stable Hamiltonian structure $\SH$ on $M$ is a vector field characterized by
\begin{enumerate}
\item $\Lambda(R_{\SH})=1$,
\item $\Omega(R_{\SH}, \cdot)=0$. 
\end{enumerate}
\end{definition}
Notice that the flow of the Reeb vector field $R_{\SH}$ preserves $\Lambda$, $\Omega$, and $g$, i.e., $\mathcal{L}_{R_{\SH}}\Lambda=0$, $\mathcal{L}_{R_{\SH}}\Omega=0$, and $\mathcal{L}_{R_{\SH}}g=0$.

\begin{definition} Let $M$ be a hypersurface in a symplectic manifold $(W, \omega)$. A \textit{stabilizing vector field} $Z$ on a neighborhood of $M$ is a vector field satisfying the following properties: 
\begin{enumerate}
\item $Z$ is transverse to $M$.
\item The 1-parameter family of hypersurfaces $M_s=\{s\} \times M$, for $s \in (-\epsilon, \epsilon)$, generated by the flow of $Z$ preserves the Reeb vector fields. That is, the Reeb vector field on each slice $M_s=\{s\} \times M$ is independent of $s$.
\end{enumerate}
\end{definition}

\begin{example}
Let $(M, \xi=\op{ker}(\lambda))$ be a contact hypersurface of a symplectic manifold $(W, \omega)$ then $\mathcal{H}=(d\lambda, \lambda)$ is a stable Hamiltonian structure in which $R_{\SH}$ is the Reeb vector field in the usual sense, the Liouville vector field $Z$ transverse to $M$ is a stabilizing vector field and $g=1$.
 \end{example}

Recall that $\Omega|_{\xi}$ is symplectic. One can then define $\Omega$-compatibility of an almost complex structure $J$ on $\R \times M$ after controlling the (end) behavior of $J$-holomorphic curves in $\R \times M$ near infinity. 

\begin{definition} An almost complex structure $J$ on $\R \times M $ is called \textit{$\Omega$-compatible} if
\begin{enumerate}
\item $J$ is $\R$-invariant.
\item  $J(\partial_s)=R_{\SH}$ and $J(-R_{\SH})=\partial_s$, where $\partial_s$ is the unit vector in the $\R$-direction.
\item $J(\xi)=\xi$.
\item $g(u, v)=\Omega(u, Jv)$ is a Riemannian metric on $\R \times M$, where $u, v \in \xi$, i.e., $J|_{\xi}$ is $\Omega$-compatible.
\end{enumerate}
\end{definition}

\begin{notation} We will denote by $\mathcal{J}(\mathcal{H})$ the space of all $\Omega$-compatible almost complex structures $J$ on $[0, \infty) \times M$.
\end{notation}

\subsection{Neck stretching} \label{NS} In this subsection we describe neck-stretching, which is a technique originating in symplectic field theory (SFT), also known as the splitting of a symplectic manifold $W$ along a contact hypersurface $M$ by deforming the almost complex structure on $W$ in some neighborhood of $M$. See \cite{BEHWZ} and \cite{EGH} for more details.\sm

Let $(M, \xi)$ be a contact type hypersurface of an exact symplectic manifold $(W, \omega=d\lambda)$ and $Z$ be a Liouville vector field positively transverse to $M$. Suppose that $M$ divides $W$ into two parts $W^+$ and $W^-$ with two boundary components $M^{+}$ and $M^-$, respectively, so that $Z$ points outwards along $M^+$ and inwards along $M^-$. \sm
 
Let $J$ be an $\omega$-compatible almost complex structure on $(W, \omega)$ which is also \textit{$(M, \lambda)$-compatible} and $Z$ be a Liouville vector field on $W$.  We wish to construct a family of $\omega$-compatible almost complex structures on $W$ as follows:\sm

Let $I_t=[-t-\epsilon, t+\epsilon]$. Denote by $\mathcal{X}_s$ the flow of the Liouville vector field $Z$. Define a diffeomorphism $\varphi_t$ as follows:
\begin{align*}
\varphi_t:I_t \times M &\rightarrow W,\\
(s, m) &\mapsto \mathcal{X}_{\beta(s)}(m),
\end{align*}
where $\beta: I_t\rightarrow [-\epsilon, \epsilon]$ is a strictly increasing function chosen in such a way that it is close to being the identity map near $s=0$. \sm

Assume that $I_t \times M$ is equipped with the symplectic form $\omega_{\beta}=d(e^{\beta(s)}\lambda)=\varphi^{*}_t (\omega)$. Let $\tilde{J}_t$ be an almost complex structure on $I_t \times M$ such that $\tilde{J}_t|_{\xi}=J|_{\xi}$.\sm

 Assume also that $\tilde{J}_t$ is invariant in the Liouville direction. One can then form
\begin{equation*}
 (I_t \times M, \tilde{J}_t) \cup_{\varphi_t} (W- \varphi_t ({I_t \times M}), J)
\end{equation*}
by gluing via the diffeomorphism $\varphi_t$. The resulting manifold is then symplectomorphic to $(W, J_t)$ where $J_t$ is the new $\omega$-compatible almost complex structure on $W$.

\begin{definition} \label{defn:NSofM}
The family $\{J_t\}_{t=0}^{\infty}$ of $\omega$-compatible almost complex structures on $W$ is called the \textit{$(M, \lambda)$-neck-stretching} of $J_0$ along a neighborhood of $M$, where $J_0$ is $(M, \lambda)$-compatible. 
\end{definition}

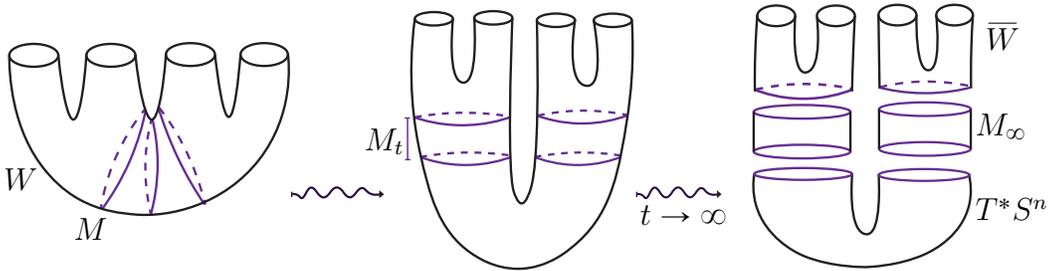
\begin{figure}[h]
\vspace{-.3cm}
\begin{overpic}[scale=1]{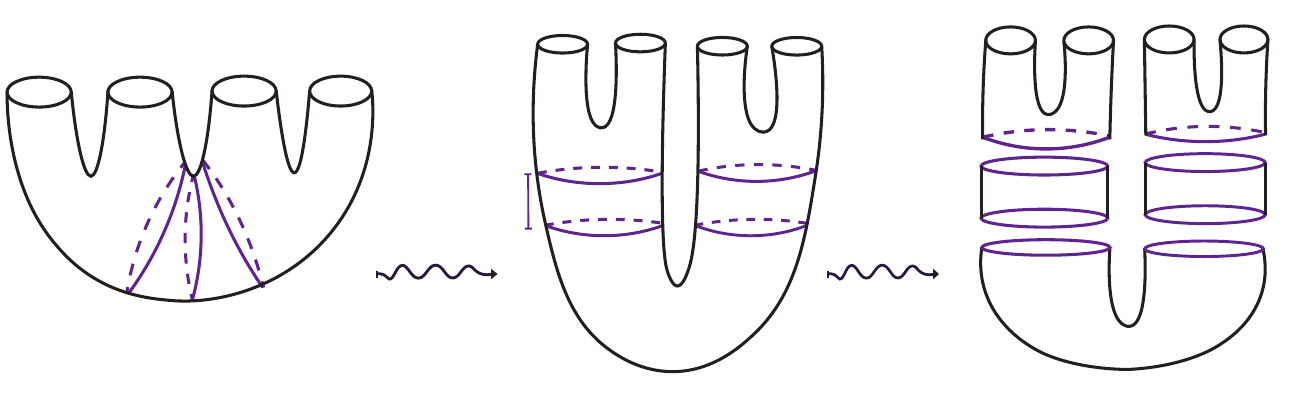}
\put(98, 24){$\widebar{W}$}
\put(97, 7){$\ct n$}
\put(0, 10){$W$}
\put(7, 5){$M$}
\put(36, 14){$M_t$}
\put(97, 15.5){$M_{\infty}$}
\put(63.5, 6.5){$t \rightarrow \infty$}
\end{overpic}
 \label{stretching}
 \caption{The illustration of the stretching of a sequence of almost complex structures along a neighborhood of a contact hypersurface $M$ until it breaks the manifold into 3 noncompact pieces $\widebar{W}$, $M_{\infty}=\R \times M$ and $\ct n$. Here $M_t$ denotes $I_t \times M$.}
\end{figure}

\subsection{Moduli spaces} \label{modulisetup} Let $(W, \omega=d\lambda)$ be a compact Liouville cobordism from $M_{+}$  to $M_{-}$. Here $M_{\pm}$ denotes the positive and negative boundarie of $W$, i.e., the Liouville vector field points outward at $M_+$ and inward at $M_-$. Let $R^{\pm}$ be the Reeb vector field for $\lambda_{\pm}=\lambda|_{M_{\pm}}$. 

The symplectic completion $(\widehat{W}, \hat{\omega})$ of $(W, \omega)$ is then the following noncompact symplectic manifold defined by attaching cylindrical ends to some collar neighborhood of $M_{\pm}$:
\begin{equation}
(\widehat{W}, \hat{\omega})=((-\infty, 0]\times M_{-}, d(e^s\lambda)) \cup_{M_{-}} (W, \omega) \cup_{M_{+}} ([0, \infty) \times M_{+}, d(e^s\lambda))
\end{equation}

\begin{definition}
An almost complex structure $J$ on $(\widehat{W}, \hat{\omega})$ is called \emph{$\hat{\omega}$-compatible} if the following conditions hold:
\begin{enumerate}
\item On the ends $([0, \infty) \times M_{+}, d(e^s \lambda))$ and $(-\infty, 0] \times M_{-}, d(e^s \lambda))$,
\begin{enumerate}
\item $J$ is $\R$-invariant.
\item $J(\partial_s)=R_{\pm}$ and $J(-R_{\pm})=\partial_s$, where $\partial_s$ is the unit vector in the $\R$-direction.
\item $J(\xi)=\xi$.
\item $J|_{\xi}$ is compatible with $d\lambda$.
\end{enumerate}
\item On $W$, $J$ is $\omega$-compatible, i.e., $g(\cdot, \cdot)=\omega(\cdot, J\cdot)$ is a Riemannian metric.
\end{enumerate}
\end{definition}

\begin{notation}We will denote by $\mathcal{J}(\widehat{W})$ the set of all $\hat{\omega}$-compatible almost complex structures on $\widehat{W}$.
\end{notation}

We can now make the following definition:

\begin{definition}\label{definition: modulispace} 
Fix nonnegative integers $k_+$ and $k_-$. Let $\gamma^{\pm}=(\gamma_1^{\pm}, \dots, \gamma_{k_{\pm}}^{\pm})$ be ordered sets of Reeb orbits of $R^{\pm}$ and let $A$ be a relative homology class in $H_2(W, {\gamma}^+ \cup {\gamma}^{-})$ where, by abuse of notation, $\gamma^{\pm}$ are viewed as subsets of $M_{\pm}$. \sm

The \textit{moduli space of $\hat{\omega}$-compatible $J$-holomorphic curves}, with $m$ marked points, homologous to $A$ and asymptotic to $(\gamma^{+}, \gamma^{-})$ in $\widehat{W}$ is a family of equivalence classes of tuples
\begin{equation*}
\mathcal{M}_{\widehat{W}, m}(J, A, \gamma^{\pm})=\{(u: (\dot{F}, j) \rightarrow (\widehat{W}, J), (z_1, \dots, z_m)) \mid  \bar{\partial}_Ju=0, u_{*}[\dot{F}]=A, z_i \in \dot{F} \} \ /{ \sim}
\end{equation*}
 with the following properties:
 \begin{enumerate}
 \item $F$ is a closed genus 0 surface.
 \item $\Gamma^{+}=(z_1^{+}, \dots, {z_{k_{+}}^{+}})$ and $\Gamma^{-}=(z_1^{-}, \dots, z_{k_{-}}^{-})$ are disjoint and ordered sets of positive and negative punctures in $F$.
 \item $\dot{F}\defeq F$ $\setminus$ $(\Gamma^+ \cup \Gamma^{-})$ is a genus 0 surface with $k$ punctures where $k=k_{+}+k_{-}$. 
\item $u: (\dot{F}, j) \rightarrow (\widehat{W}, J)$ is an asymptotically cylindrical $J$-holomorphic curve, i.e., $u$ is asymptotic to $\R \times \gamma_i^{\pm}$ near each $z_i^{\pm} \in \Gamma$ for $i=1, \dots, k_{\pm}$ and $du \circ j = J \circ du$.
 \item  $u: (\dot{F}, j) \rightarrow (\widehat{W}, J)$ represents the relative homology class $A \in H_2(W, \tilde{\gamma}^+ \cup \tilde{\gamma}^{-})$.
 \end{enumerate}
\end{definition}

Here the equivalence is given by $$(u, (z_1, \dots, z_m)) \sim (u \circ \phi, (\phi^{-1}(z_1), \dots, \phi^{-1}(z_m)))$$ where $\phi$ is an automorphism of $(\dot{F}, j)$ taking $\Gamma^+$ to $\Gamma^+$ and $\Gamma^-$ to $\Gamma^-$.

We explain what it means for $u: (\dot{F}, j) \rightarrow (\widehat{W}, J)$ to represent $A$.\sm

Let $\bar{F}$ be the compact surface with boundary which is obtained from $\dot{F}$ by adding $\{\pm \infty\} \times S^1$ to each cylindrical end. Let $r: \widehat{W} \rightarrow W$ be the retraction onto $W$ such that $r|_{W}$ is the identity map and $r(s, m)=m \in M_{\pm}$ when $(s, m) \in (-\infty, 0] \times M_{-}$ or $[0, \infty) \times M_{+}$. Then $r \circ u: \dot{F} \rightarrow W$ has a continuous extension $\bar{u}:(\bar{F}, \partial{\bar{F}})\rightarrow (W, \tilde{\gamma}^+ \cup \tilde{\gamma}^{-})$ and by $u$ representing a nonzero relative homology class $A=[u] \in H_2(W, \tilde{\gamma}^+ \cup \tilde{\gamma}^{-})$, we mean $\bar{u}$ representing $A$ whose image under
\begin{equation*}
H_2(W, \tilde{\gamma}^+ \cup \tilde{\gamma}^{-}) \stackrel{\partial}\longrightarrow H_1(\tilde{\gamma}^+ \cup \tilde{\gamma}^{-})
\end{equation*}
is given as
\begin{equation*}
\partial A= \sum \limits_{i=1}^{k_+} [\gamma _i^+]+\sum \limits_{i=1}^{k_-} [\gamma _i^-] \in H_1(\tilde{\gamma}^+ \cup \tilde{\gamma}^{-}).
\end{equation*}

The moduli space $\mathcal{M}_{\widehat{W}, m}(J, A, \gamma^{\pm})$ then comes with a natural map that records the evaluation of the curves at the marked points.

\begin{definition}The\textit{ evaluation map} of $\mathcal{M}_{\widehat{W}, m}(J, A, \gamma^{\pm})$ is the map
\begin{align*}
\text{ev}: \mathcal{M}_{\widehat{W}, m}(J, A, \gamma^{\pm}) &\rightarrow \widehat{W}^m,\\
(u, (z_1, \dots, z_m))&\mapsto (u(z_1), \dots, u(z_m)).
\end{align*}
\end{definition}

The evaluation map encodes relations between the topology of $\widehat{W}$ and the structure of the moduli space of curves in $\widehat{W}$. The degree of evaluation map is an algebraic count of the number of $J$-holomorphic curves with a marked point passing through a generic point in $\widehat{W}$ that carry the relative homology class $A$. \sm

In dimension 4, there are particularly powerful results regarding the intersection properties of $J$-holomorphic curves.

\begin{thm}\cite[Theorem 2.6.3, Positivity of intersections]{MS} $(X, J)$ Let $X$ be a 4-dimensional almost complex manifold and $u_1$ and  $u_2$ be two simple $J$-holomorphic curves in $X$ representing the homology classes $A_1$, $A_2 \in H_2(X)$. Each intersection of the images contributes a positive integer to the homological intersection $A_1 \cdot A_2$ of the two curves and that number is 1 if and only if the intersection is transverse.
\end{thm}

In particular, the homological intersection $A_1 \cdot A_2$ is always nonnegative. If it is 0, then $u_1$ and $u_2$ are disjoint $J$-holomorphic curves. We also note that the positivity of intersections is a local result and applies to punctured $J$-holomorphic curves in the symplectization.

\begin{thm}\cite[Theorem 2, Automatic transversality]{W4} Let u be an embedded punctured $J$-holomorphic curve in the symplectization $(\R \times M, J)$ of a 3-manifold $M$ equipped with a stable Hamiltonian structure and generic almost complex structure $J$. Then every curve in the unparametrized moduli space  $\widetilde{\mathcal{M}}_{\R \times M, m}(J, A, \gamma^{\pm})$ is embedded.
\end{thm}

\section{Iterated planar open book decompositions} \label{sec:IPLF} In this section we introduce the notions of an iterated planar Lefschetz fibration, an iterated planar open book decomposition and related necessary background to introduce these two notions.
\subsection{Open book decompositions} 
\begin{definition} \label{definition: AOBD}
An \textit{abstract open book decomposition} is a pair $(F, \Phi)$, where:
\begin{enumerate}
\item $F$ is a compact $2n$-dimensional manifold with boundary, called the \textit{page} and
\item $\Phi: F \rightarrow F$ is a diffeomorphism preserving $\partial F$, called the \emph{monodromy}.
\end{enumerate} 
\end{definition}

\begin{definition}An \emph{open book decomposition} of a compact $(2n+1)$-dimensional manifold $M$ is a pair $(B,\pi)$, where:
\begin{enumerate}
\item $B$ is a codimension 2 submanifold of $M$ with trivial normal bundle, called the \emph{binding} of the open book.
\item $\pi: M - B \rightarrow S^{1}$ is a fiber bundle of the complement of $B$ and the fiber bundle $\pi$ restricted to a neighborhood $B \times \D$ of $B$ agrees with the angular coordinate $\theta$ on the normal disk $\D$.
\end{enumerate}
\end{definition}

The preimage $F_{\theta}:= \overline{\pi^{-1}(\theta)}$, for $\theta \in S^1$, gives a $2n$-dimensional manifold with boundary $\partial F_{\theta}= B$ called the \emph{page} of the open book. The holonomy of the fiber bundle $\pi$ determines a conjugacy class in the orientation-preserving diffeomorphism group of a page $F_{\theta}$ fixing its boundary, i.e., in $\mbox{Diff}^{+}(F_{\theta}, \partial F_{\theta})$ which we call the {\em{monodromy}}.\sm

Since $M$ is oriented, pages are naturally co-oriented by the canonical orientation of $S^1$ and hence are naturally oriented. Also, the binding inherits an orientation from the open book decomposition of $M$. We assume that the given orientation on $B$ coincides with the boundary orientation induced by the pages. \sm

By using this description, one can construct a closed $(2n+1)$-dimensional manifold $M$ from an abstract open book in the following way: Consider the mapping torus 
\begin{equation*}
F_{\Phi}=[0,1] \times F \bigm/ (0, \Phi(z)) \sim (1, z). 
\end{equation*}

The boundary of $F_{\Phi}$ is then given by
\begin{equation*}
\partial F_{\Phi}=[0, 1] \times \partial F \bigm/ (0, z) \sim (1, z),
\end{equation*}
since $\Phi(z)=z$ on $\partial F$. Let $|\partial F_{\Phi}|$ denote the number of boundary components of $F_{\Phi}$. We set
\begin{equation*}
M_{(F, \Phi)}= F_{\Phi}\cup \left( \bigsqcup_{|\partial F_{\Phi}|} \mathbb{D} \times \partial F \right) 
\end{equation*}
via the following identification:
\begin{equation*}
\left((\theta, p) \in \partial F_{\Phi}\right) \sim \left((\theta, p) \in \partial \D \times \partial F\right).
\end{equation*}

Here  the boundary of each disk $S^1 \times \{pt\}$ in $\mathbb{D} \times \partial F$ gets glued to $\{pt\} \times S^1$ in the mapping torus. The abstract open book $(F, \Phi)$ is then an open book decomposition of a closed $(2n+1)$-dimensional manifold $M$ if $M_{(F, \Phi)}$ is diffeomorphic to $M$. \smallskip

The mapping torus $F_{\Phi}$ carries the structure of a natural fibration $F_{\Phi} \rightarrow S^1$, away from $\partial F$, whose fiber is the interior of the page $F$ of the open book decomposition. 

\begin{definition}
 A contact structure $\xi$ on a compact manifold $M$ is said to be \emph{supported by an open book} $(B,\pi)$ of $M$ if it is the kernel of a contact form $\lambda$ satisfying the following:
\begin{enumerate}
\item $\lambda$ is a positive contact form on the binding and
\item $d\lambda$ is positively symplectic on each fiber of $\pi$.
\end{enumerate}
If these two conditions hold, then the open book $(B,\pi)$ is called a \emph{supporting open book} for the contact manifold $(M, \xi)$ and the contact form $\lambda$ is said to be \emph{adapted} to the open book $(B,\pi)$. Furthermore, the contact form $\lambda$ is called the \textit{Giroux form}.
\end{definition}

 \subsection{Iterated planar Lefschetz fibrations}
 
\begin{definition} A \emph{topological Lefschetz fibration} is a smooth map $f: W \rightarrow \D$, where $W$ is a compact manifold of dimension $2n$ with corners and $\D$ is a 2-disk, with the following properties:
\begin{enumerate}
\item The critical points of $f$ are isolated, nondegenerate, and are in the interior of $W$.
\item If $p\in W$ is a critical point of $f$, then there are local complex coordinates $(z_{1},\dots , z_{n})$ about $p=(0, \dots, 0)$ on $W$ and $z$ about $f(p)$ on $\D$ such that, with respect to these coordinates, $f$ is given by the complex map $z = f(z_1, \dots, z_n) = z_{1}^{2} +\dots +z_{n}^{2}$.
\item There exists a decomposition of $\partial W$ into horizontal and vertical parts $\partial_h W$ and $\partial_v W$, respectively, which meet at a codimension 2 corner and $\partial_v W=f^{-1}(\partial \D)$ and $\partial_h W=\cup_{z\in \D} \partial F_z$. Here $\partial_h W$ is the boundary of all fibers including singular ones.
\end{enumerate}
\end{definition}
Denoting the critical points by $p_1, \dots, p_r \in W$ and the corresponding critical values by $c_1, \dots, c_r \in \D$, $r\geq 1$, for each $z \in \D \backslash \{c_1, \dots, c_r \}$, $F_{z} = f^{-1}(z)$ is a $(2n-2)$-dimensional submanifold of $W$ with nonempty boundary.\sm

\begin{figure}[h]
\vspace{-.7cm}
\begin{overpic}[scale=.7]{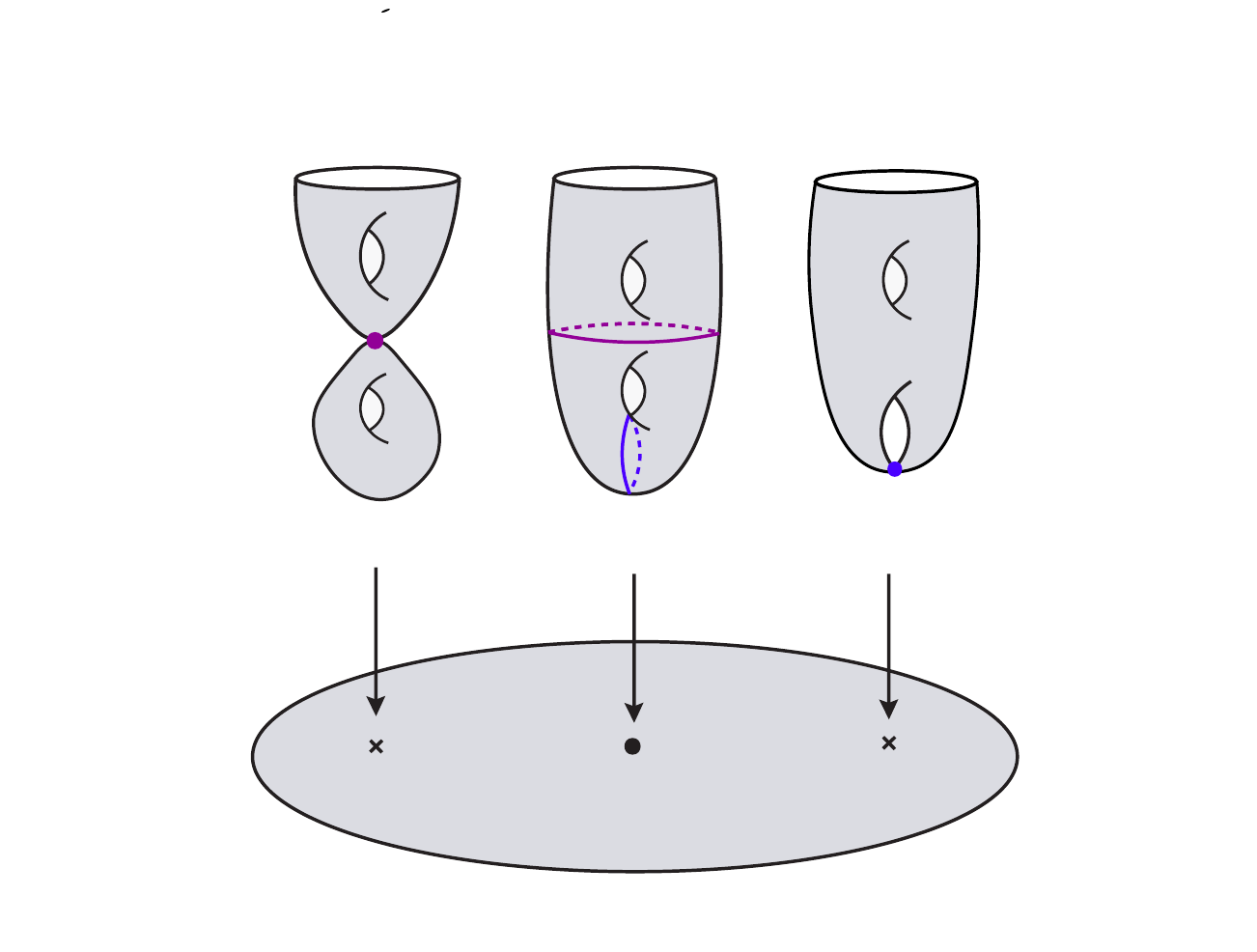}
\put(50,39){$ $}
\put(13, 40){$W$}
\put(13, 15){$\D$}
\put(28, 12){$c_1$}
\put(49.5, 11.5){$z$}
\put(69, 12){$c_2$}
\end{overpic}
\vspace{-.4cm}
\caption{The figure representing the Lefschetz fibration $f: W \to \D$ with critical values $c_1$ and $c_2$. Here the purple and blue curves represent vanishing cycles of the regular fiber over $z$. The fibers with purple and blue dots represent singular fibers.}
 \label{figure: LeF}
\end{figure}
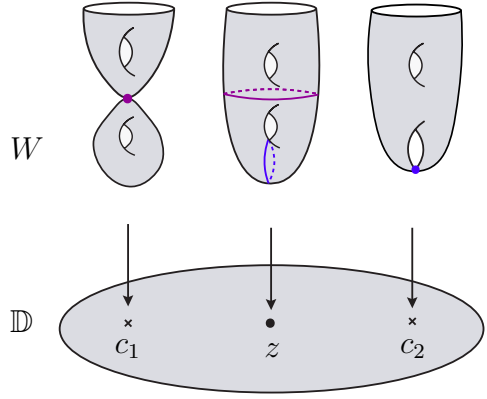

In order to talk about an open book decomposition induced by the boundary restriction of the Lefschetz fibration $f$, one should smooth the corners of $\partial W$ to begin with. The Lefschetz fibration $f$ on $W$ will then induce an open book decomposition on $\partial W$. \sm

We shall now describe the compatibility conditions between a topological Lefschetz fibration $f: W \rightarrow \D$ and a symplectic form $\omega$ on $W$ as given in \cite{KS}. 

\begin{definition} A \emph{symplectic Lefschetz fibration} is a pair $(f: W \rightarrow \D, \omega=d\lambda)$, where \begin{enumerate}
\item $f$ is a topological Lefschetz fibration.
\item $\omega=d\lambda$ is an exact symplectic form on $W$ such that $(\partial W, \lambda|_{\partial W})$ is a contact manifold.
\item Each regular fiber $F_z$, $z\in \D \setminus \{c_1, \dots, c_r\}$, carries the structure of an exact symplectic manifold with contact type boundary. 
\item Let $F_{z_0}$ be a reference fiber and denote by $\lambda|_{\partial F_{z_0}}$ the pullback from $\partial F_{z_0}$ to $\D \times ([0, 1] \times \partial F_{z_0})$. Then, on a neighborhood $([0, 1] \times \partial (\partial_v W)) \times \partial_h W$ of $\partial_h W$, the contact form $\lambda=\lambda|_{\partial F_{z_0}}+Kf^{*}(\sigma)$, where $\sigma$ is the standard Liouville form on $\D$ and $K$ is a sufficiently large constant.
\end{enumerate}
\end{definition}

Note that the condition (4) in the definition above implies that $\omega=d\lambda$ is an exact symplectic form and the vector field defined by $\lambda$ is transverse to $\partial W$ and points outwards (see McLean  \cite[Theorem 2.15]{M}). Therefore, one can smooth the corners of $\partial W$ to obtain a symplectic deformation of $W$.\sm

We will next examine the compatibility of an almost complex structure on with a symplectic Lefschetz fibration. Given a symplectic Lefschetz fibration $(f: W \rightarrow \D, \omega=d\lambda)$. Fix a complex structure $j$ on $\D$. We can then make the following definition:

\begin{definition}An \textit{almost complex structure $J$ on $W$ is compatible with a symplectic Lefschetz fibration} $(f: W \to \D, \omega=d\lambda)$ if 
\begin{enumerate}
\item $f$ is $(J, j)$-holomorphic.
\item $J=J_0$ and $j=j_0$ near critical points of $W$, where $J_0$ and $j_0$ are the standard almost complex and complex structures on $\C ^2$ and $\C$, respectively.
\item In a neighborhood of $\partial_h W$, $J$ is a product almost complex structure.
\item In a neighborhood of $\partial_v W$, $J$ is invariant in the radial direction.
\item $J|_{T_hW}$ is $\omega$-compatible, where $T_hW$ is the horizontal component of the tangent bundle $TW$.
\end{enumerate}
\end{definition}
Note that the matrix for $J$ never achieves to be antisymmetric. Therefore, $J$ is not $\omega$-compatible. However, $J$ is tamed by $\omega$ (see Seidel \cite[Lemma 2.1]{PS}). Denote by $\mathcal{J}_f(W)$ the space of almost complex structures compatible with a symplectic Lefschetz fibration. Moreover, it is known that $\mathcal{J}_f(W)$ is nonempty and contractible. 

\begin{definition}
A \emph{completion of the symplectic Lefschetz fibration} $(f: W \to \D, \omega=d\lambda)$ is a pair $(\hat{f}: \widehat{W}\rightarrow \C, \widehat{\omega}=d\widehat{\lambda})$, where 
\begin{enumerate}
\item The completion of a neighborhood of $\partial_h W$ is defined by $\widehat{W}_v \defeq  ([1, \infty) \times \partial_h W) \cup W$.
\item The contact form $\lambda|_{\partial F_{z}}$ extends over the cylindrical end $[1, \infty) \times \partial_h W$ to $s \cdot \lambda|_{\partial F_z}$ where $s$ denotes the $\R^{+}$-direction.
\item Denote by $\partial_v(\widehat{W}_v)$ the vertical boundary of $\widehat{W}_v$. The completion of $W$ is defined by $\widehat{W}^= \widehat{W}_v \cup ( \partial_v(\widehat{W} _v) \times [1, \infty))$.
\item The extension of $f$ over $Y \defeq  \partial_v(\widehat{W}_v) \times [0, \infty)$ is $\hat{f}|_{Y}(v, t)=\hat{f}|_{ \partial_v(\widehat{W}_v)}(v)$.
\item The contact form $\widehat{\lambda}={\lambda}|_{\hat{F}} + K\hat{f}^{*}(\widehat{\sigma})$, where ${\lambda}|_{\hat{F}}$ is the pullback from $\partial F$ to $\D \times ([1, \infty) \times \partial F)$,  $\widehat{\sigma}$ is the standard Liouville form on $\C$, and $K$ is a sufficiently large constant. 
\end{enumerate}
\end{definition}

Now we will turn our attention to the almost complex structures on the completion of a symplectic Lefschetz fibration. Let $J$ be a compatible almost complex structure on $\widehat{W}$ and $j$ be the standard complex structure on $\C$.

\begin{definition} \label{definition:symp_comp_acs}
An \textit{almost complex structure $J$ on $ \widehat{W}$ is compatible with} $(\hat{f}: \widehat{W} \rightarrow \C, \widehat{\omega})$ if the following conditions hold:
\begin{enumerate}
\item $d\hat{f} \circ J= j \circ d{\hat{f}}$.
\item On $([1, \infty) \times \partial F) \times \D$, $J$ is a product almost complex structure.
\item On $\partial_v W \times [1, \infty)$, $J$ is invariant in the radial direction.
\item $J$ is $\lambda$-compatible on the completion of a neighborhood of $\partial_h W$.
\item On $W$, $J$ is $\omega$-compatible, i.e., $g(\cdot, \cdot)=\omega(\cdot, J\cdot)$ is a Riemannian metric.
\end{enumerate}
\end{definition}

Denote by $\mathcal{J}^h(\widehat{W})$ the space of all almost complex structures compatible with $\hat{f}$. It is known that $\mathcal{J}^h(\widehat{W})$ is nonempty and contractible (c.f. \cite[Section 2.2]{PS}).

\begin{definition}
A Lefschetz fibration $f: W \rightarrow \D$ with dim$_\R W=4$ is called \textit{planar} if all fibers of $f$ have genus zero, i.e., are planar surfaces.
\end{definition}

Note that the boundary restriction of a planar Lefschetz fibration induces a planar open book decomposition. Now we are ready to introduce a new class of Lefschetz fibrations for a Weinstein domain $W$ as follows:

\begin{definition} A Weinstein domain $(W^{2n}, \omega)$, $n\geq 2$, admits an \textit{iterated planar Lefschetz fibration} if 
\begin{enumerate}
\item there exists a sequence of symplectic Lefschetz fibrations $f_2, \dots, f_{n}$ where $f_{i}: W^{2i} \rightarrow \D$ for $i=2, \dots, {n}$.
\item Each regular fiber of $f_{i+1}$ is the total space of $f_{i}$, i.e., $W^{2i}$ is a regular fiber of $f_{i+1}$.
\item $f_{2}: W^4\rightarrow \D$ is a planar Lefschetz fibration.
\end{enumerate}
Here the superscript $2i$ indicates the dimension of $W^{2i}$. Notice that when $n=2$, $W^4$ admits a planar Lefschetz fibration.
\end{definition}

Let us give here two important examples of manifolds admitting iterated planar Lefschetz fibrations.

\begin{example}
For $n\geq 2$, the unit disk bundle $W^{2n}=\ct n$ admits an iterated planar Lefschetz fibration since each regular fiber of a Lefschetz fibration on $\ct n$ is $\ct {n-1}$ and the Lefschetz fibration on $\ct 2$ is planar with fibers $\ct 1=[0, 1] \times S^1$.
\end{example}

\begin{example} Consider $A_k$-singularity which is, by definition, symplectically identified with
$$\{(z_1, \dots, z_n) \mid z_1^2+\dots+z_{n-1}^2+z_{n}^{k+1}=1\} \subset (\C^n, \omega_{std})$$ for $n\geq 3$ and $k\geq2$. It is known that $A_k$-singularity can be expressed as a plumbing of $k$ copies of $\ct {n-1}$. Notice that each regular fiber of a Lefschetz fibration on $A_k$-singularity is $\ct {n-1}$. Hence, the discussion above implies also that $A_k$-singularity admits an iterated planar Lefschetz fibration.
\end{example}

If $f: W \rightarrow \mathbb{D}$ is an iterated planar Lefschetz fibration, then, after smoothing the corners, the boundary of $W$ inherits an open book decomposition whose pages are diffeomorphic to the regular fibers of $f$. The following definition is motivated by looking at the open book decomposition induced by the boundary restriction of an iterated planar Lefschetz fibration. 

\begin{definition}An open book decomposition of a contact manifold $(M^{2n+1}, \xi)$ whose pages admit iterated planar Lefschetz fibrations is called an \textit{iterated planar open book decomposition} and finally a contact manifold supporting such an open book decomposition is called an \textit{iterated planar contact manifold}.
\end{definition}

We shall now construct an almost complex structure $J$ on $\widehat{W}$ compatible with a suitable stable Hamiltonian structure on $\widehat{W}$ to be able to study the $J$-holomorphic curves on the pages of the open book decomposition of $\partial W$. Denote by $(W, \omega=d\lambda)$ a $2n$-dimensional Weinstein manifold with boundary $\partial W=M$. Suppose that $(M, \xi=\op{ker}\lambda)$ bounds an iterated planar Lefschetz fibration. Then the following lemma holds:
 
\begin{lemma} \label{lemma:wendl}
Let $(f: W \to \D, \omega=d\lambda)$ be an iterated planar Lefschetz fibration. Then there exists an almost complex structure $J \in \mathcal{J}^h(\widehat{W})$ compatible with a suitable stable Hamiltonian structure $\mathcal{H}$ on $[0, \infty) \times M$ such that the fibers of $\hat{f}: \widehat{W} \rightarrow \C$ are holomorphic. 
\end{lemma}

\begin{proof} Recall that the space of almost complex structures compatible with the completion of a symplectic manifold (c.f. Definition \ref {definition:symp_comp_acs}) is nonempty and contractible. Then, there exists at least one $J \in \mathcal{J}^h(\widehat{W})$ defined as in Definition \ref{definition:symp_comp_acs}. It is left to show that it is also compatible with some suitable stable Hamiltonian structure on the symplectization of $M$ such that the fibers of $\hat{f}$ are holomorphic.\sm 

We will start with the construction of a suitable stable Hamiltonian structure on $M$ and then examine the abstract open book decomposition of $M$. The advantage of working with stable Hamiltonian structures is that if we pick a suitable stable Hamiltonian structure as in \cite{W2}, then then it admits arbitrarily small perturbations that are contact structures supported by the planar open book decomposition defined on the boundary of a planar Lefschetz fibration. \sm

Let $\theta$ and $\phi$ be the coordinates on the boundary of fibers and on the boundary of $\D$, respectively. Denote by $s$ and $t$ the vertical and horizontal collar coordinates of $\partial_v W$ and $\partial_h W$, respectively. Let $F_z$ be the fiber over $z \in D$. Consider the fiberwise Liouville form $\lambda|_{\partial F_z}$ on each of the fibers such that 
\begin{enumerate}
\item $d\lambda|_{F_z}$ is a symplectic form and
\item $\lambda|_{\partial F_z}$ is a contact form.
\end{enumerate}

Note that the fiberwise Liouville form $\lambda|_{\partial F_z}$ is independent of the choice of $z \in \D$. Observe also that $\lambda|_{\partial F_z}=e^{s} d\theta$ near $\partial_h W$. Let $\sigma$ be the standard Liouville form on $\D$ such that $\sigma=e^t d\phi$ is a Liouville form near the boundary of $\D$. By using this information, one can construct a Liouville form near corners whose flow smooths the corners of $\partial W$.\sm

\begin{figure}[h]
\begin{overpic}[scale=1.2]{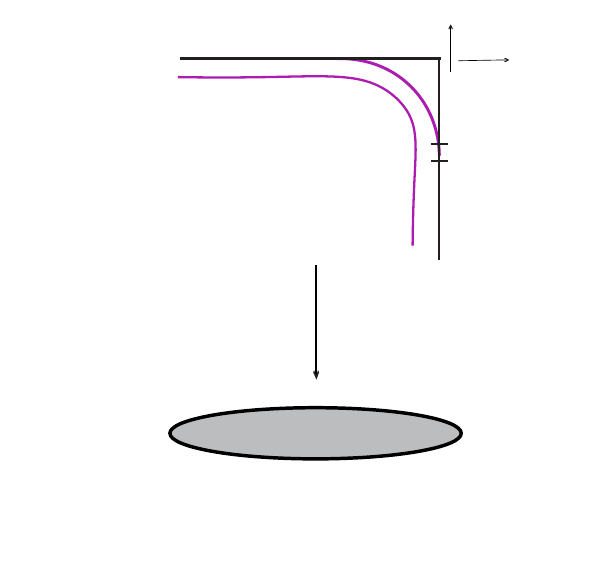}
\put(55,40){$f$}
\put(74,91){$s$}
\put(87,82){$t$}
\put(50,52){$W$}
\put(68,79.5){$M$}
\put(60,74){$M_{\epsilon}$}
\put(75,70){$a$}
\put(75,64){$b$}
\put(30,86){$\partial_h W$}
\put(75,51){$\partial_v W$}
\put(21,20){$\D$}
\end{overpic}
\vspace{-1cm}
\caption{The local picture representing the Lefschetz fibration $f$ after the smoothing the corner of the total space $W$. Here the hypersurface $M$ represents the smoothed boundary $\partial W$ and $M_{\epsilon}$ denotes another hypersurface below and sufficiently close to $M$.}
\label{smoothingcorner}
\end{figure}
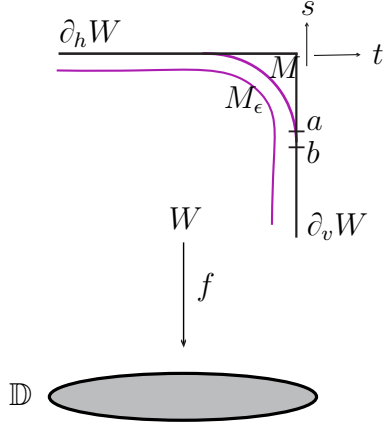

We shall now construct a Liouville vector field near the corners by turning the fiberwise Liouville form into a Liouville form:
\begin{equation}
\lambda_K=\lambda|_{\partial F_z}+Kf^{*} \sigma,
\end{equation}
where $K$ is a sufficiently large positive constant and $\omega_K=d\lambda_K$ is symplectic (see Seidel \cite[Lemma 1.5]{PS}). Note that the Liouville form $\lambda_K$ induces a Liouville vector field $Z_K=(\partial_t+\partial_{s})$ near the corner and a Reeb vector field $\partial_{\phi}$.\sm

Now we want to define a new vector field $Z$ in some neighborhood of $M$ as follows:

 \[ Z :=\begin{cases} 
      Z_K & s \geq a \text{ and near }\partial_h W, \\
      \partial_t+\beta(s)\partial_{s} & b\leq s \leq a, \\
      \partial_t & \text{for } s \leq b, 
   \end{cases}
\] where $\beta$ is a suitably chosen cutoff function.\sm

In light of this vector field, one can construct a stable Hamiltonian structure $\mathcal{H}$ induced by $Z$ on $M$ as follows:
\begin{align*}
\Lambda:= & {\iota _z\omega_K}|_{M}, \\
\Omega:= & {\omega_K}|_{M}.
\end{align*}

Recall that the characteristic vector field is given by the $\phi$-direction. Notice that the flow of $Z$ doesn't change that line field since there is no dependence on the $s$-direction. Now observe that if we take the smooth hypersurface $M$ and flow it along $Z$, we obtain a 1-parameter family of hypersurfaces whose characteristic line fields are all identical. In other words, $Z$ is a stabilizing vector field. Hence, if we choose another hypersurface $M_\epsilon$ which is transverse to the Liouville vector field $Z_K$, lying below $M$ and also sufficiently close to it as in Figure \ref{smoothingcorner}, then the region in between is a symplectic cobordism. \sm

Observe also that on $\partial_h W$, $\omega|_{\partial_h W}=Ke^t dt\wedge d\phi+d\lambda|_{\partial_h W}=d\lambda|_{\partial_h W}$. Near the corner, $d\lambda|_{\partial F_z}=e^{s}d{s}\wedge d\theta$. Note that this construction is true for any $(2n+1)$-dimensional Weinstein fillable contact manifold $M$. \sm

The remainder of the proof closely follows the construction in \cite{W2}. Denote by $(F, \Phi)$ the abstract open book decomposition of $M$. Let $F_\Phi$ and $B$ the associated mapping torus and the binding, respectively. Here we are using the notation in Definition \ref{definition: AOBD}. By page, we mean the fibers of the mapping torus $F_{\Phi}$. The compatibility check of $J$ on $M$ will be done on $F_\Phi$ and $\D \times B$ separately. \sm

On $F_\Phi$, the co-oriented hyperplane distribution $\xi$ is tangent to the pages and is preserved under any $J$ we choose. Therefore, for any choice of $J$, the fibers of the mapping torus, i.e., the pages become holomorphic. That is to say, $\{s\} \times \{\theta\} \times F$ in $\R \times F_\Phi$ is a $J$-holomorphic curve for any choice of $J$. Now we need to check if that extends to a neighborhood of the binding. To do this, we need to find a $J$-holomorphic cylinder asymptotic to $\{r=0\}$ near the binding that fit smoothly with the fiber $\{\theta\} \times F$. To see this, we refer the reader to \cite{W2} for a detailed analysis near the binding.\sm

One can then attach the cylinder to the holomorphic fiber $\{s\} \times \{\theta\} \times F$ which extends to a holomorphic fiber in the symplectization of $F_{\Phi}$ to an embedded $J$-holomorphic curve in the symplectization of the Weinstein fillable contact manifold $M$. Therefore, $J\in \mathcal{J}^h(\widehat{W})$ is compatible with the symplectization of the stable Hamiltonian structure $\mathcal{H}$.
\end{proof}

\section{Proof of the Main Theorem}
The proof is based on an inductive argument which helps us carry several results and phenomena particular to dimension $4$, such as automatic transversality and positivity of intersection, to higher dimensions. 

\subsection{Setup} \label{section:setup}

Let $f_i: W^{2i} \to \D$ be an iterated planar Lefschetz fibration and let $M=M^{2n+1}$ and $W=W^{2i}$ with $\partial W^{2i}=M^{2i-1}$. Let $(\mathbb{R}\times M^{2n+1}, d(e^s\lambda))$ be the symplectization of a contact manifold $(M^{2n+1}, \xi=\op{ker} \lambda)$ supporting an open book whose pages are $2n$-dimensional Weinstein domains $W^{2n}$, where $s$ denotes the $\mathbb{R}$-coordinate. Assume also that $W^{2n}$ admits an iterated planar Lefschetz fibration.\sm

Denote by $\widehat{W}$ the completion of a Weinstein domain $W$. Let $J \in \mathcal{J}^h(\widehat{W})$. We will specify the Reeb orbits $\gamma^+$ and the relative homology class $A$ for the moduli space $\mathcal{M}_{\widehat{W}, m}(J, A, \gamma^+)$ of planar $J$-holomorphic curves in $\widehat{W}$ asymptotic to $\gamma^+$. Notice that these curves do not have any negative ends since $W$ is a Weinstein domain.\sm 

When $n=1$, we take $\gamma^+$ to be the binding of the open book supported by $M^3$. Similarly, when $n=2$, we take $\{0\} \times \gamma^+$ in the neighborhood $\D \times \partial W^4=\D \times M^3$ of the binding of the open book of $M^5$. \sm

More generally, when $n>1$, the associated Reeb orbits for $M^{2n+1}$ are $$\{0\} \times \{0\} \times \dots \times \{0\} \times \gamma^+ \subset \D \times \D \times \dots \times \D \times M^3 \subset \D \times M^{2n-1}$$ which we will simply call $\gamma$ throughout the paper. \sm

We set $A=[\{0\} \times \{0\} \times \dots \times \{0\} \times W^2] \in H_2(W^{2n}, \gamma)$ where $$\{0\} \times \{0\} \times \dots \times \{0\} \times W^2 \subset \D \times W^{2n-2}.$$

\begin{notation} In the rest of this paper, we write $\mathcal{M}_{\widehat{W}}(J)$ instead of $\mathcal{M}_{\widehat{W}, m}(J, A, \gamma)$. We will assume $m=1$ unless stated otherwise.
\end{notation}

\begin{prop} \label{main_prop} 

Suppose $W^{2n}$ admits an iterated planar Lefschetz fibration. Then there exists an almost complex structure $J \in  \mathcal{J}^h(\widehat{W}^{2n})$ compatible with a suitable stable Hamiltonian structure $\mathcal{H}$ such that $\mathcal{M}_{\widehat{W}^{2n}}(J)$ is regular and $\widehat{W}^{2n}$ is ``filled" by planar $J$-holomorphic curves, i.e. the degree of the evaluation map ev: $\mathcal{M}_{\widehat{W}^{2n}}(J) \to \widehat{W}^{2n}$ is 1 mod 2.

\end{prop}

The proof can be summarized as follows: Denote by $f_n: W^{2n} \to \D$ the associated iterated planar Lefschetz fibration. We will study the foliation of $\widehat{W}^{2n}$ by planar $J$-holomorphic curves by using $f_n$. If $f_n$ has no singular points, then we can use the facts that there exists an almost complex structure $J$ compatible with a suitable stable Hamiltonian structure such that each regular fiber $W^4$ of $f_3$ is foliated by planar $J$-holomorphic curves \cite[Main Theorem]{W2} and each regular fiber of $f_n$ is $W^{2n-2}$. Hence, the filling follows.\sm

If $f_n$ has singular points, we need to make sure that the holomorphic curves pass through the singular points safely without any obstruction. Away from singular points, regular fibers are foliated by planar $J$-holomorphic curves by using the inductive argument provided by the iterated planar Lefschetz fibration structure. However, around each singular point, it is not clear if we can extend the family of $J$-holomorphic curves in a regular fiber across the singular fibers safely. The singular point may block the holomorphic curves from passing through the singular point. \sm

A priori, we have no control over curves when they get pinched. We want to extend the family of curves in the regular fiber across singular fibers. In order to extend across the singular fibers, we will apply neck stretching to a neighborhood of a critical point since each curve in that neighborhood will then look like limiting Reeb orbits. By neck stretching, we are able to normalize the neighborhood of the critical point. Hence, the picture becomes clearer. This local analysis is to make sure that there exists a planar $J$-holomorphic curve through each point in $\widehat{W}^{2n}$. \sm

We will prove the Proposition \ref{main_prop} via an inductive construction examining the following arguments inductively:
\begin{enumerate}
\item All $J$-holomorphic curves in the associated moduli space are transversely cut out, i.e., the moduli space is regular.
\item The compactification of the associated moduli space is nice. This will be elaborated in Section \ref{subsec:compactify}.
\item The associated moduli space has the expected dimension.
\end{enumerate}

Once we achieve these for all $n\geq2$, we can use the evaluation map to show that there exists a unique holomorphic curve, up to algebraic count, passing through each generic point in $\widehat{W}^{2n}$ since the evaluation map is a degree $1$ map. The rest of the proof will then be reduced to applying neck-stretching to a neighborhood of a contact hypersurface in $\R \times M^{2n+1}$ which will help us observe that all $J$-holomorphic curves in the moduli space $\mathcal{M}_{\R \times M, 1}(J)$ are asymptotic to closed Reeb orbits. 

\subsection{Index calculation} \label{indexformula}
Let $u \in \mathcal{M}_{\widehat{W}^{2n}}(J)$ be a planar $J$-holomorphic curve in $\widehat{W}^{2n}$.
\begin{lemma}
The index of u (the Fredholm index of the linearized Cauchy-Riemann operator) in $\mathcal{M}_{\widehat{W}^{2n}}(J)$ is $2n-2$.
\end{lemma}

\begin{proof}
Recall that the expected dimension of the moduli space of $J$-holomorphic curves in $\widehat{W}^{2n}$ is the Fredholm index of $\bar{\partial}_J u$. This index is
\begin{equation}\label{indformula2}
\mbox{ind}(u)=((n-3)\chi (\Sigma)+\sum \limits_{\gamma^+} \mu_{\text{CZ}}(\gamma ^+, \tau)-\sum \limits_{\gamma^-}\mu_{\text{CZ}}(\gamma^-, \tau)+2c_1(u^{*}T(\widehat{W}^{2n})),
\end{equation} 
where $\gamma^+$ and $\gamma^-$ are positive and negative ends of $u({\dot{F}})$, respectively and $\tau$ is the trivialization of $\xi$ along $\gamma^{\pm}$. Note that the Conley-Zehnder index is independent of the choice of $\tau$. The dimension of the ambient manifold is $2n$, which explains the term $(n-3)$.\sm

Let $k$ be the number of positive ends of $u(\dot{F})$. Recall that we do not have any negative ends. Therefore, $\mu_{\text{CZ}}(\gamma^-, \tau)=0$. Let us begin with the case $n=2$. With respect to the trivialization $\tau$ for which boundary of the pages are longitudinal, the Conley-Zehnder index of each of the Reeb orbits is $\mu_{\text{CZ}}(\gamma^+, \tau)=1$ for the elliptic case. \sm

The relative first Chern number of the bundle $u^{*}T(\widehat{W}^{2n}) \to \dot{F}$ with respect to the trivialization $\tau$ is
$$
c_1(u^{*}T(\widehat{W}^{2n}))=c_1(T\dot{F} \oplus N_u)=c_1(T\dot{F})+c_1(N_u)
$$
where $N_u$ denotes the normal bundle to $u$. Observe also that transverse direction is trivial. Thus, we have $$c_1(u^{*}T(\widehat{W}^{2n}))=c_1(T\dot{F}).$$

It remains to analyze the tangential component of $c_1$. Note that the relative Chern class of a planar surface $\dot{F}$ with $k$ positive ends is given by
\begin{equation*}
<c_1(T\dot{F}, \tau), \dot{F}>=<e(\dot{F}), \dot{F} >=\chi(\dot{F})=2-k.
\end{equation*}

Having determined the relative first Chern number, we shall compute the index for $n=2$: 
\begin{align*}
\text{ind}(u)&=(2-3)(2-0-k)+k-0+2(c_1(T\dot{F}))\\
&=-1(2-k)+k+2(2-k)=2.
\end{align*}

Now consider the case when $u: (\dot{F}, j) \longrightarrow (\widehat{W}^{2n}, J)$ and dim$_{\R}(\widehat{W}^{2n})=2n$. Consider also a projection map from a neighborhood $\mathcal{N}(\partial (W^{2n-2}))$ of $\partial (W^{2n-2})$ to $\D$. This map gives a trivialization $\tau$ of $\mathcal{N}(\partial (W^{2n-2})) \cong \partial (W^{2n-2}) \times \D$. We then obtain $$\mu_{\operatorname{CZ}}(\gamma^+, \tau)=\dfrac{2n}{2}-1=n-1.$$ Thus, Equation \ref{indformula2} implies that  
\begin{align*}
\text{ind}(u)&=(n-3)\chi(\dot{F})+k\cdot (n-1)-0+2\chi(\dot{F})\\
&=(n-3)(2-k)+k\cdot (n-1)+2(2-k)\\
&=(n-1)(2-k)+k\cdot (n-1)\\
&=2n-2.
\end{align*}
\end{proof}


\subsection{Induced complex structure on $T^*S^{n-1}$} \label{inducedcpx}  
Consider the iterated planar Lefschetz fibration $$f_n: W^{2n} \to \D.$$ 

One can then realize the cotangent bundle $\ct {n-1}$ as a subset of a regular fiber near a singular fiber of $f_n$. Equivalently, $\ct {n-1}$ can be thought of as the standard neighborhood of the vanishing cycle $S^{n-1}$ of a fiber of $f_n$.\sm

In this subsection, we consider $$T^*S^{n-1} =\{(u,v)\in \mathbb{R}^{n} \times \mathbb{R}^{n} \mid |u|=1 \ \mbox{and} \ u \cdot v=0\} \subset \R^{2n}$$ and identify it with a fiber of the map
\begin{equation} \label{eqn1}
\begin{split}
\varphi _n: \mathbb{C}^n &\rightarrow \mathbb{C}, \\
(z_1, \dots, z_{n}) &\mapsto z_1^2+ \dots +z_{n}^2.
\end{split}
\end{equation}

This identification gives a complex structure on $\ct {n-1}$ which is induced from $\C^n$ with the standard complex structure. We then prove:

\begin{lemma}\label{lemma:inducedcpx}
The standard complex structure on $\C^n$ induces a complex structure on $\ct {n-1}$ which is asymptotically cylindrical at the ends of $\ct {n-1}$.
\end{lemma}

\begin{proof}
Consider the map $\varphi _n$ given by Equation \ref{eqn1}. We want to identify $\varphi^{-1}_n (1)$ with $\ct {n-1}$. Let $w$ be a positive real number and $x=(x_1, \dots, x_n)$ and $y=(y_1, \dots, y_n)\in \R^n$. Then consider $$ \varphi^{-1}_n (w) =\{x+iy  \in \mathbb{C}^{n} \mid |x|^2 - |y|^2 =w \ \mbox{and} \  x\cdot y=0 \}. $$ 

Note that $\varphi^{-1}_n (w)$ is a symplectic submanifold of $(\C^n, \omega)$ with the standard symplectic form $\omega=\sum \limits_{j=1}^{n} dx_j\wedge dy_j$.

Consider $T^*S^{n-1} =\{(u,v)\in \mathbb{R}^{n} \times \mathbb{R}^{n} \mid |u|=1 \ \mbox{and} \ u \cdot v=0\} \subset \R^{2n}.$
Note that $\ct {n-1}$ is symplectic submanifold of $\R^{2n}$ with the restriction of the standard symplectic form $dv\wedge du$ on $\R^{2n}$. 

Consider the map
\begin{align*}
\phi : \C^n & \rightarrow \R^{2n},\\
x+iy &\mapsto \left(-\dfrac{x}{|x|}, y|x|\right).
\end{align*}

Denote by $\Phi$ the restriction map $\phi|_{\varphi^{-1}_n (1)}$. Then the fiber $\varphi^{-1}_n (1)$ can be identified with $\ct {n-1}$ via the following diffeomorphism:
\begin{align*}
\Phi : \varphi^{-1}_n (1) & \rightarrow T^*S^{n-1},\\
x+iy &\mapsto \left(-\dfrac{x}{|x|}, y|x|\right).
\end{align*}

Denote by $\lambda_1=\frac{1}{2}(xdy-ydx)|_{\varphi^{-1}_n (1)}$ and $\lambda_2=vdu|_{\ct {n-1}}$ the Liouville forms of $\varphi^{-1}_n (1)$ and $T^*S^{n-1}$, respectively. Note that $\lambda_1$ and $\lambda_2$ are pullbacks of the standard primitives of the symplectic forms on $\C ^n$ and $\R ^{2n}$ to $\varphi^{-1}_n (1)$  and $\ct {n-1}$, respectively. 

\begin{claim} The map $\Phi: (\varphi^{-1}_n (1), \lambda_1) \to (T^*S^{n-1}, \lambda_2)$ is an exact symplectomorphism.
\end{claim} 

\begin{proof} In light of the construction above, it is sufficient to show that $\Phi^{*}(\lambda_2)-\lambda_1$ is exact.
\begin{align*}
\Phi^{*}(vdu) &= \Phi^{*}(\sum_{j=1}^n v_jdu_j),\\
&=\sum_{j=1}^n y_j |x|\left(-\frac{1}{|x|}dx_j+\sum_{l=1}^n -\frac{x_l x_j}{|x|^2}dx_j \right),\\
&=\sum_{j=1}^n (-y_j dx_j+\sum_{l=1}^n -\frac{x_l x_j y_j}{|x|^2}dx_j ),\\
&=\sum_{j=1}^n -y_j dx_j=-ydx.\\
\Phi^{*}(vdu)-\frac{1}{2}(xdy-ydx) &= -ydx-\frac{1}{2}(xdy-ydx)=-\frac{1}{2}\left(xdy+ydx\right)=-\frac{1}{2}d(x\cdot y)=0.
\end{align*}

Hence, $\Phi: \varphi^{-1}_n (1) \to T^*S^{n-1}$ is an exact symplectomorphism. 
\end{proof}

Via the diffeomorphism $\Phi: \varphi^{-1}_n (1) \to T^*S^{n-1}$, the complex structure on $\C^n$ then induces a complex structure on $\ct {n-1}$. Let $j$ be a complex structure on $T^*S^{n-1}$ induced from the standard complex structure on $\C^n$ via the diffeomorphism $\Phi: \varphi^{-1}_n (1) \to T^*S^{n-1}$.\\

\noindent\textit{End behavior of $j$.} In this subsection, we consider the singular fiber $Q_s$ given by 
\begin{equation*}
Q_s=\left\{\begin{array}{ll}
  |x|^2-|y|^2=0, \\
  x \cdot y=0
\end{array}
\right\}
\end{equation*} 
and identify it with $Q_s\cong\varphi^{-1}_n (0)-\{0\} \to T^*S^{n-1}-\{\text{zero section}\}$. \sm

 Next we will analyze the behavior of the induced complex structure $j$ when it goes sufficiently far out in the fiber to see if it is cylindrical at the ends. To see this, we will examine the induced complex structure $j$ on the singular fiber and regular fiber, separately.\sm

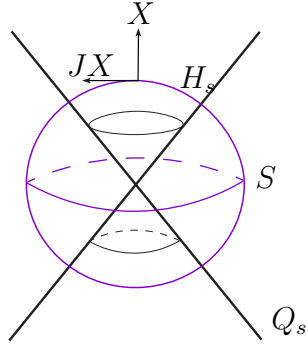
\begin{figure}[h]
\begin{overpic}[ scale=.5]{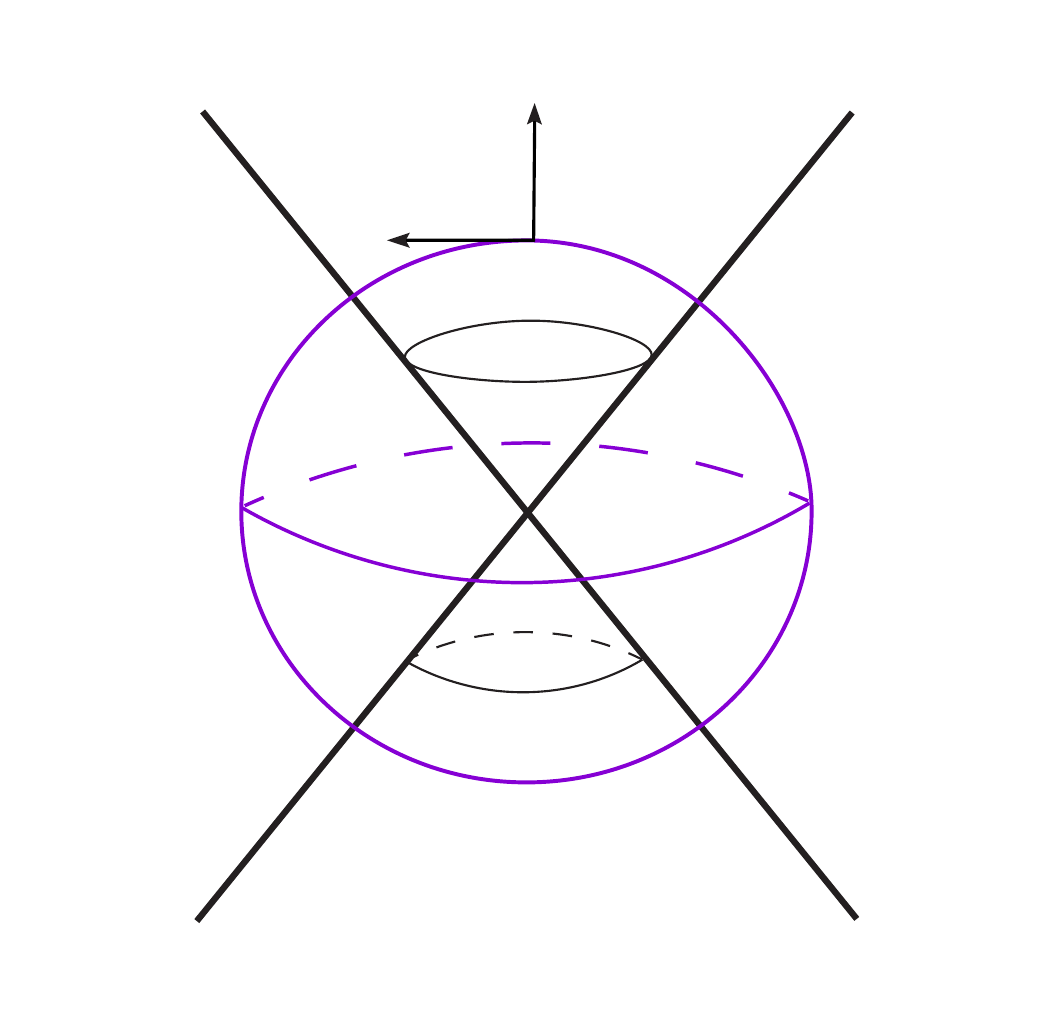}
\put(48,88){$X$}
\put(61,72){$H_s$}
\put(33,75){$JX$}
\put(80,47){$S$}
\put(84,12){$Q_s$}
\end{overpic}
\caption{Local picture representing the intersection of the singular fiber $Q_s$ with a sphere $S$ above a critical point. Here $H_s$ represents the hypersurface obtained by the intersection of $Q_s$ with the sphere $S$ given by the equation $x^2+y^2=2c^2$ for $c\in \R$.}
\label{singendbehavior}
\end{figure}

\begin{claim}
 The induced complex structure $j$ is cylindrical on $Q_s$.
\end{claim}
\begin{proof}
The singular fiber $Q_s$ is equipped with the standard symplectic form $\omega=d\lambda$ where $\lambda=\frac{1}{2}(xdy-ydx)$ and the induced complex structure $j$. Denote by $X$ the Liouville vector field given by $X=\frac{1}{2}\left(x\frac{\partial}{\partial x}+y\frac{\partial}{\partial y}\right)$. To observe the end behavior of the induced complex structure, we will intersect $Q_s$ with some sphere depicted as in Figure \ref{singendbehavior} and define the hypersurface, which we will call $H_s$, as follows:
\begin{equation*}
H_s=\left\{\begin{array}{lll}
  |x|^2-|y|^2=0, \\
  x \cdot y=0,\\
  |x|^2+|y|^2=4.
\end{array} 
\right\}
\end{equation*}

\begin{enumerate}

\item  We first show that the induced complex structure $j$ sends the Liouville vector field $X$ to the Reeb vector field $jX$.\\

One can compute that the orthogonal complement to $T_{(x, y)}H_s$ at $(x, y) \in H_s$ is spanned by $(x, -y), (y, x), (x, y)$ or equivalently, $(x, 0), (0, y), (y, x)$. Note that the induced complex structure is the standard complex structure on $\C^n$. Therefore, one can compute the vector field $jX$ as follows:
\begin{equation*}
jX=j\left(\frac{1}{2}\left(x\frac{\partial}{\partial x}+ y\frac{\partial}{\partial y}\right)\right)=\frac{1}{2}\left(-y\frac{\partial}{\partial x}+x\frac{\partial}{\partial y}\right).
\end{equation*}

Notice that $jX=(-y, x)$ is tangent to the hypersurface $H_s$ for all $(x, y) \in H_s$. This follows from the following computation.
\begin{align*}
(-y, x) \cdot (x, 0) &= -y \cdot x = 0\\ 
(-y, x) \cdot (0, y) &= x \cdot y = 0\\
(-y, x) \cdot (y, x) &= |x|^2-|y|^2 = 0
\end{align*}

Moreover, $jX$ is the Reeb vector field for $\lambda$. To see this, 
\begin{align*}
i_{jX} \omega |_{H_s}= \omega(jX, \cdot)|_{H_s} &= \omega \left(\frac{1}{2}\left(-y\frac{\partial}{\partial x}+x\frac{\partial}{\partial y}\right), \cdot\right)|_{H_s}\\
&= \frac{1}{2}\left(-ydy-xdx\right)|_{H_s}\\
&= -\frac{1}{4} d(|x|^2+|y|^2)|_{H_s}= 0.
\end{align*}

It also satisfies the rescaling condition:
\begin{equation*}
\lambda(jX)=\frac{1}{2}(xdy-ydx)\left(\frac{1}{2}\left(-y\frac{\partial}{\partial x}+x\frac{\partial}{\partial y}\right)\right)=\frac{1}{4}(|x|^2+|y|^2)=1
\end{equation*}
since $|x|^2+|y|^2=4$ on $H_s$. Thus, the induced complex structure $j$ sends $X$ to the Reeb vector field $jX$.\\

\item Next we show that the contact distribution $\xi$ is preserved by $j$.\\

Firstly, note that $(jX)^{\perp}=\R<(x, 0), (0, y), (y, x), (-y, x)>^{\perp}$. For $(x, y) \in H_s$, we know that the tangent bundle $TH_s \perp \R<(x, 0), (0, y), (y, x)>$ and $\xi \subset TH_s$. Therefore, for any $(u,v) \in \xi$, we have 
\begin{align}
(u,v) \cdot (x, 0)&=u \cdot x=0, \label{eq1} \\
(u,v) \cdot (0, y)&=v \cdot y=0,\label{eq2} \\
(u,v) \cdot (y, x)&=u\cdot y+ v \cdot x=0. \label{eq3}
\end{align}  

Also, $\op{ker}(\lambda)$ is perpendicular to the Reeb vector field $jX$. Therefore, $\xi$ is spanned by the orthogonal complement of $(x, 0), (0, y), (y, x),$ and $(-y, x)$. That is,
\begin{equation*}
\xi=\R<(x, 0), (0, y), (y, 0), (0, x)>^{\perp}
\end{equation*}
when restricted to the hypersurface $H_s$. In addition to the relations above , we have 
\begin{eqnarray}
(u,v) \cdot (y, 0)=u \cdot y=0, \label{eq4}\\
(u, v) \cdot (0, x)=v\cdot x=0. \label{eq5}
\end{eqnarray}

Next we claim that this implies that the contact distribution is preserved by $j$, i.e., $j(\xi)=\xi$. To see this, let $(u, v) \in \xi$. Then $j(u, v)=(-v, u)$ and 
\begin{align*}
(-v, u) \cdot (x, 0)&=-v\cdot x=0 \text{ by Equation \ref{eq5}}\\
(-v, u) \cdot (0, y)&= u \cdot y=0 \text{ by Equation \ref{eq4}}\\
(-v, u) \cdot (y, 0)&= -v\cdot y=0 \text{ by Equation \ref{eq2}}\\
(-v, u) \cdot (0, x)&= u \cdot x=0 \text{ by Equation \ref{eq1}}
\end{align*}

Therefore, $(-v, u) \in \xi$. With that, we have verified that $\xi$ is preserved by the induced complex structure. 
\end{enumerate}

Therefore, when $H_s$ flows with respect to $X$, the induced complex structure $j$ becomes cylindrical at the ends of the singular fiber $Q_s$ which proves the claim.
\end{proof}

It is left to study the regular fiber case. Consider the regular fiber
\begin{equation*}
Q=\left\{\begin{array}{ll}
  |x|^2-|y|^2=1, \\
  x \cdot y=0
\end{array}
\right\}
\end{equation*}
near $Q_s$ and identify it with $\ct {n-1}$.
Unlike the singular case, the induced complex structure $j$ on $Q$ is not automatically cylindrical at the ends of $Q$. It is rather asymptotically cylindrical which follows along the same lines as \cite[Example 2]{B} by using a cylindrical coordinate chart on the symplectization of the unit sphere bundle $S\ct {n-1}$.\sm

This finishes the proof of Lemma \ref{lemma:inducedcpx}.
\end{proof}

\subsection{Regularity of $J$-holomorphic curves}\label{subsec:regularity} 
Here we recall the space $ \mathcal{J}^h(\widehat{W}^{2n})$ of almost complex structures compatible with the Lefschetz fibration $\hat{f}_n: \widehat{W}^{2n} \rightarrow \C$.

\begin{prop}\label{prop:regular}
There exists an almost complex structure $J \in \mathcal{J}^h(\widehat{W}^{2n})$ compatible with a suitable stable Hamiltonian structure $\mathcal{H}_0$ such that 
\begin{enumerate}
\item $\mathcal{H}_0$ is as desired in Lemma \ref{lemma:wendl}.
\item $J$ is an almost complex structure compatible with the symplectic connection away from the singular fibers such that the regular fibers of $f_n$ are $J$-holomorphic.
\item Each element in $\mathcal{M}_{\widehat{W}^{2n}}(J)$ lies in a fiber of $\hat{f}_n$.
\item $\mathcal{M}_{\widehat{W}^{2n}}(J)$ is regular. 
\item In a neighborhood of a singular fiber of $f_n$, $J$ is given by the gluing of the almost complex structure in a neighborhood of a critical point and a cylindrical product almost complex structure in the complement of the neighborhood of the critical point in the neighborhood of the singular fiber.
\end{enumerate}
\end{prop} 

\begin{proof}  
The proof is structured in two parts. In Section \ref{constructionofJ}, we will construct an almost complex structure in $\mathcal{J}^h(\widehat{W}^{2n})$ that makes the fibers of $f_n$ $J$-holomorphic. In Section \ref{pfofregularity}, we will show that the almost complex structure we obtained in the first part of the proof makes the moduli space $\mathcal{M}_{\widehat{W}^{2n}}(J)$ regular. Both parts concern analysis around a critical point and a regular point of $f_n$ along with an extension analysis to the rest of the total space $W^{2n}$ of $f_n$.

\subsubsection{Construction of the suitable almost complex structure}\label{constructionofJ} We will start with the case $n=2$. Then there exists an almost complex structure on $W^4$ compatible with a suitable stable Hamiltonian structure such that $W^4$ is foliated by finite-energy planar $J$-holomorphic curves \cite[Main Theorem]{W2}. \sm

Next we will consider the case $n >2$ in two parts:\sm

\noindent \underline{Part I: Regular fibers}\sm

Recall that $f_n$ is an iterated planar Lefschetz fibration. Then there exists a compatible almost complex structure on each regular fiber $W^{2n-2}$ of $f_n$ such that $W^{2n-2}$ is $J$-holomorphic and is filled by planar $J$-holomorphic curves. Let $p \in \D$ be a regular value away from critical values of $f_n$ and consider a neighborhood $N_p$ of $p$. Above $N_p$, the almost complex structure is a product almost complex structure given by the almost complex structure on the fiber $W^{2n-2}$ and the lift of the complex structure on the base.\sm

\noindent \underline{Part II: Singular fibers}\sm

Recall that in a neighborhood of a critical point, the Lefschetz fibration $f_n: W^{2n} \rightarrow \D$ is given by the complex map 
\begin{align*}
\C^n &\rightarrow \C \\
(z_1, \dots, z_n) & \mapsto z_1^2+\dots+z_n^2
\end{align*}

By Lemma \ref{lemma:inducedcpx}, we know that the standard complex structure on $\C^n$ induces an asymptotically cylindrical complex structure in a neighborhood of a critical point. Therefore, $f_n$ is a holomorphic map around each critical point. See Section \ref{inducedcpx} for a more comprehensive discussion on the induced complex structure. \sm

Let $E_c$ be a neighborhood of a critical value $c \in \D$ of $f_n$ and $E=\ct {n-1}$ be a neighborhood of the critical point above $E_c$ such that $f_n^{-1}(E_c)-E$ looks like a product manifold. We can then pick a cylindrical product almost complex structure $J^c$ on the complement $f_n^{-1}(E_c)-E$ such that the the complement looks like a product almost complex manifold. \sm

The almost complex structure $j$ on $E$ can be described as follows: Away from a neighborhood of a critical point in $E$, the almost complex structure is the complex structure induced from $\C^n$ as described in Lemma \ref{lemma:inducedcpx}. At the critical point, it is given by a product almost complex structure as described in Section \ref{pfofregularity}.\sm

Since $j$ and $J^c$ have the same end behavior, the almost complex structure on each singular fiber can be obtained by the gluing $J^c \# J_R \# j$, where $J_R$ is a cylindrical almost complex structure on the neck region connecting $E$ with $f_n^{-1}(E_c)- E$. This concludes the existence of an almost complex structure on singular fibers of $f_n$.\\

Let $c_1, \dots, c_m$ be critical values of $f_n$ and $N_p$ be a neighborhood of a regular value $p\in \D$.

\begin{claim} The almost complex structure $J$ extends to $\hat{f}_n^{-1}(\D-\bigcup\limits_{s=1}^m E_{c_s}-N_p)$.
\end{claim}
\begin{proof}
Let $N=\hat{f}_n^{-1}(\D-\bigcup\limits_{s=1}^m E_{c_s}-N_p)$. To study this extension, we will examine the almost complex structure given by the transverse and fiber directions of $N$. To extend the regularity to the rest of the total space, we slowly interpolate between the regular almost complex structures we obtained on the regular and singular fibers. The change in almost complex structures in the transverse direction is almost constant when zoomed in. Hence, the almost complex structure is close to being a product structure since the variation in the transverse direction is negligible. \sm

Notice that the transverse direction, i.e. symplectic connection direction, of $N$ is given by the pullback of the complex structure on $\D$. To study the fiber direction, we recall that the space of almost complex structures compatible with the symplectic form on the fibers is contractible. Therefore, the extension in the fiber direction is automatic. Hence, the almost complex structure on a regular fiber is homotopic to the almost complex structure on the singular fiber. 

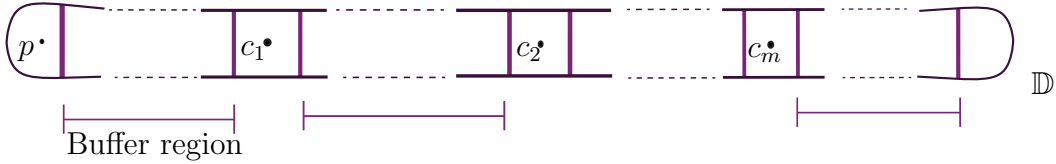
\begin{figure}[h]
\hspace{-1.2cm}
\begin{overpic}[scale=1.3]{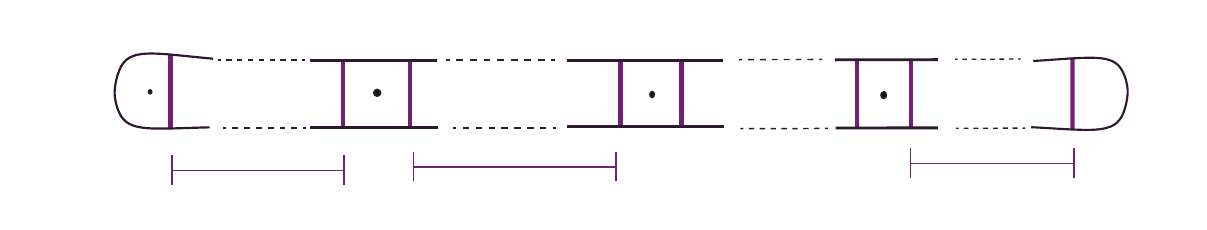}
\put(10.5,10.3){$p$}
\put(29,10.3){$c_1$}
\put(52,10,2){$c_2$}
\put(71.3,10,1){$c_m$}
\put(14.5, 2){Buffer region}
\put(95,7){$\D$}
\end{overpic}
\caption{Each $c_i \in \D$ represents a nondegenerate critical value for $i=1,\dots, m$ and $p \in \D$ is a regular value.}
 \label{buffer}
\end{figure}

Denote by $J^{n-1}$ the almost complex structure on $\widehat{W}^{2n-2}$ for $n \geq 2$. The almost complex structure on $$\hat{f}_{n}^{-1}(N_p) \cong \widehat{W}^{2n-2} \times N_p$$ will then look like a product almost complex structure $(J^{n-1}, j_{\D})$, where $j_{\D}$ is the pullback almost complex structure of the standard complex structure on $\D$.\sm

Consider the buffer region $[r_1, r_2] \times [0, 1]$ connecting the neighborhoods $N_p$ and $E_{c_1}$. Let $t \in [r_1, r_2]$ and $r_1-r_2 \gg 0$. Then the almost complex structure above the buffer region will look like the family $(J^{n-1}_t, j_{\D})$ given by the almost complex structure on the fiber of $\hat{f}_{n}$ and the pullback complex structure $j_{\D}$ since 
$$\hat{f}_{n}^{-1}([r_1, r_2] \times [0, 1])=\widehat{W}^{2n-2} \times [r_1, r_2] \times [0, 1].$$ 

Here, the almost complex structure switches from $J^{n-1}_{r_1}$ to $J^{n-1}_{r_2}$ as we move from $\{r_1\} \times [0, 1]$ to $\{r_2\} \times [0, 1]$. Notice that the change in $J^n$ in the $t$-direction is very slow above the buffer region. When we zoom in the buffer region, nothing changes locally. Hence, the almost complex structure on a regular fiber of $\hat{f}_n$ extends to $\hat{f}_n^{-1}(\D-\bigcup\limits_{s=1}^m E_{c_s}-N_p)$.
\end{proof}

Finally, one needs to check that no degenerations of curves, such as breaking, occur as we are changing the almost complex structure during this interpolation. 

\subsubsection{Proof of the regularity} \label{pfofregularity} 

Observe that the 1-parameter family of almost complex structures $\{J_t\}_{t\in[r_1, r_2]}$ is regular as a family since it is generic. However, at some point $t_{*} \in [r_1, r_2]$, $J_{t_{*}}$ might not be regular. At such points, there might be breaking of curves. In what follows, we argue that no such degeneration occurs. \sm

Consider the map $\hat{f}_3: \widehat{W}^6 \to \C$. Observe that the cylindrical ends in the fiber direction containing the Reeb orbits map down to a point. Hence, by the open mapping theorem, every end containing Reeb orbits has to live in a fiber since $\hat{f}_n$ is a holomorphic map.\sm

\begin{figure}[h]
\begin{overpic}[scale=1.2]{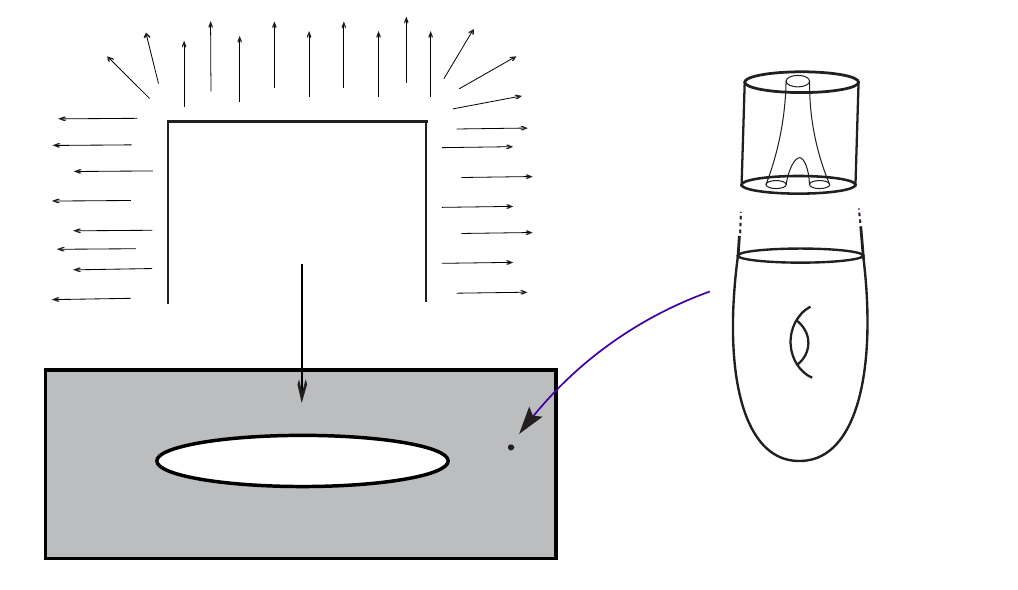}
\put(51, 5){$\C$}
\put(48, 13.5){$z$}
\put(86, 17){$\hat{\phi}^{-1}(z)$}
\put(77.5, 44){$u$}
\put(86, 45){$\R \times \partial W^{2n}$}
\put(20.5, 11.5){$S^1$}
\put(3, 48){$\R \times M$}
\put(26,26){$\hat{\phi}$}
\end{overpic}
\caption{The figure on the left is the map $\hat{\phi}: \R \times M \to \C$. The figure on the right is the fiber $\hat{\phi}^{-1}(z)$ over $z \in \C$ and a 2-level building depicting the breaking of curves. Here $u$ represents the end of a broken curve containing the Reeb orbit in $\widehat{W}^{2n}$.}
\label{breaking2}
\end{figure}
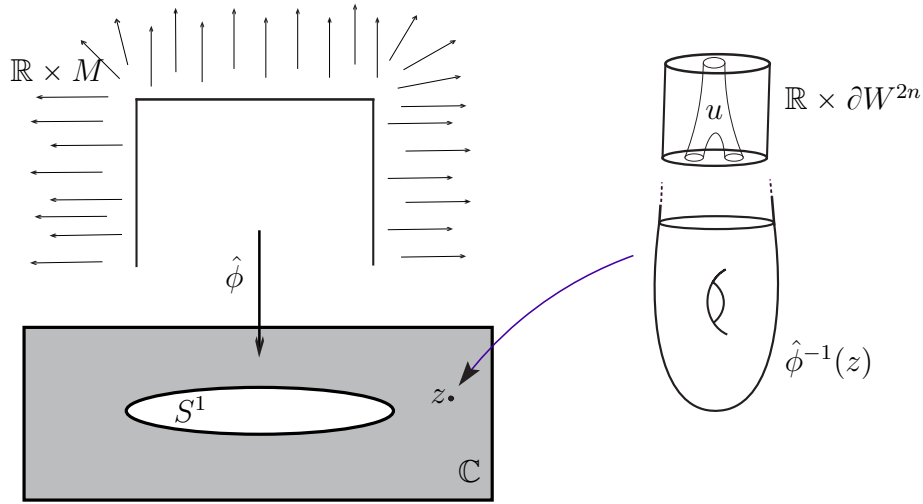

Let $u$ be a curve in the cylindrical end $\R \times \partial W^4$ containing the Reeb orbit depicted as in Figure \ref{breaking2}. Hence, by the discussion above, $u$ has image inside $W^4$ and $W^4$ admits a planar Lefschetz fibration whose fibers are planar surfaces with no negative ends. Therefore, no such $u$ exists unless it is a trivial cylinder and trivial cylinders are ignored in holomorphic buildings by convention.\sm

Observe that one can apply this construction to $W^{2n}$ for all $n \geq 2$ inductively since the map $\hat{f}_n$ is a projection map from the completion of $W^{2n}$ to a point for all $n$, Thus, we could go one dimension down to use the information that $u$ is contained in a fiber. This concludes that there exists no degenerations of curves as we are changing the almost complex structure during the interpolation mentioned in the claim above.\sm

Next we will examine the regularity of curves in a neighborhood of a regular fiber $W^{2n-2}$ lying above $N_p$. Recall that all $J$-holomorphic curves in $W^{2n-2}$ are transverse due to the inductive argument. Hence, the associated moduli space is regular away from critical points.\sm

The goal is now to construct an almost complex structure and do a similar regularity analysis as above in some neighborhood $E$ of a critical point. To do this, we will use the local model where $W^{2n}$ is identified with $\C^n$ and $f_n:W^{2n}\rightarrow \D$ has only one critical point. The local picture will then look like a conic bundle as in Figure \ref{conicbundle}. Note that one can extend this model case to the whole fibration by extending the base so that the base includes all critical points.
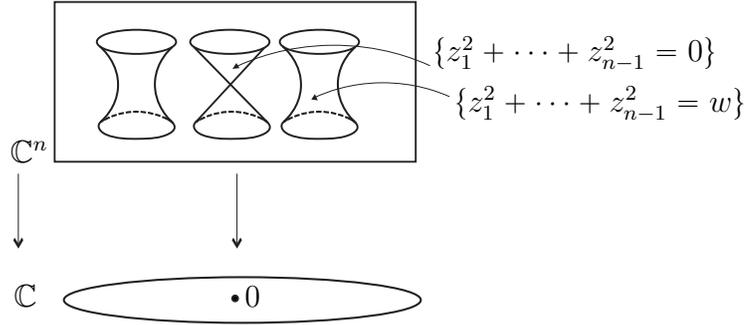
\begin{figure}[h]
\begin{overpic}[scale=.8]{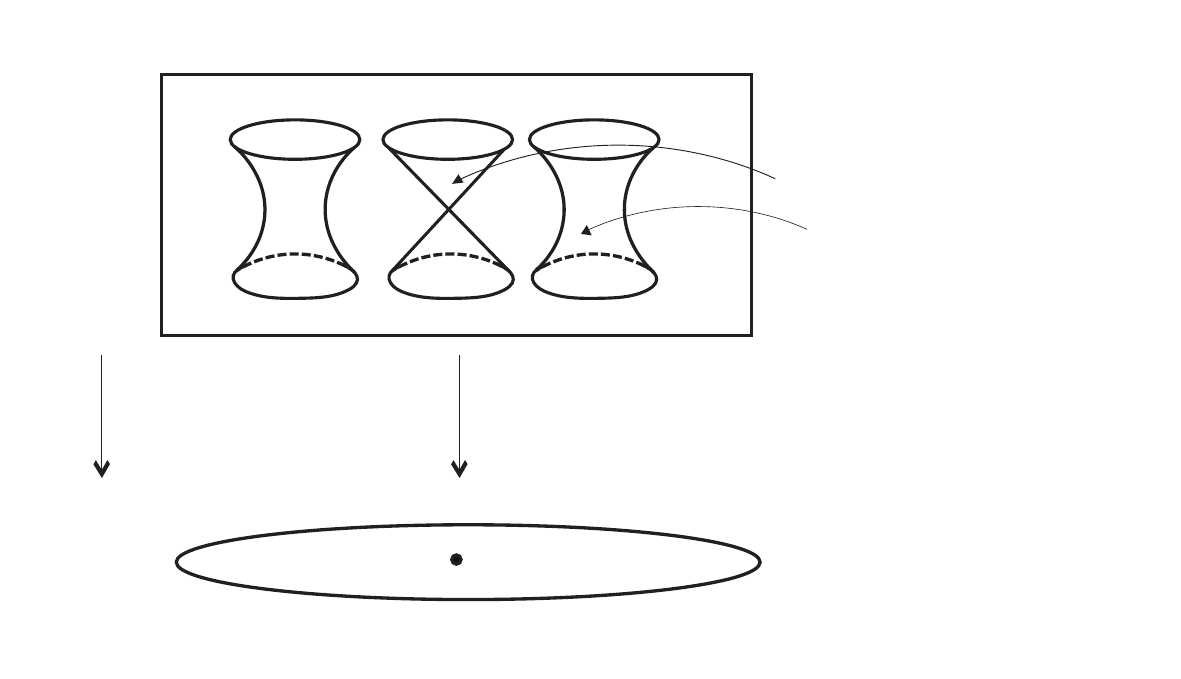}
\put(7.5,29){$\C^n$}
\put(8,9){$\C$}
\put(66,43){$\{z_1 ^2+\dots +z_{n-1} ^2=0\}$}
\put(69,36){$\{z_1 ^2+\dots +z_{n-1} ^2=w\}$}
\put(40,9){$0$}
\end{overpic}
\caption{Local picture representing the behavior of vanishing cycles above a small neighborhood of a critical point.}
\label{conicbundle}
\end{figure}

Let us begin with the case $n=3$. Consider the iterated planar Lefschetz fibration $f_3:W^6\rightarrow \D$, where each regular fiber $W^4$ of $f_3$ is foliated by planar $J$-holomorphic curves. We would like to understand how, under ideal conditions, the foliation of $W^6$ degenerates as each regular fiber $W^4$ of $f_3$ degenerates into a singular fiber. \sm

To see this, consider the following projection map
\begin{align*}
\varphi _3: \mathbb{C}^3 &\to \mathbb{C}, \\
(z_1, z_2, z_3) &\mapsto w=z_1^2+ z_{2}^2 +z_{3}^2,
\end{align*}
where $\varphi^{-1}_3 (w) =\{x+iy  \in \mathbb{C}^{3} \mid |x|^2 - |y|^2 =w \ \mbox{and} \  x\cdot y=0 \}$. One can similarly define the following projection map 
\begin{align*}
\pi: \C^3 &\to \C^2,\\
(z_1, z_2, z_3) &\mapsto (z_3, w),
\end{align*}
where $w=z_1^2+ z_{2}^2 +z_{3}^2$. Note that $\pi^{-1}(z_3, w)=\{z_1 ^2+z_2 ^2=w-z_3 ^2\}$ and $\pi^{-1}(0, 0)=\{z_1^{2}+z_2^2=0\}$. Thus, around each critical point of $f_3$, we have the local picture shown in Figure \ref{fibers}. \sm

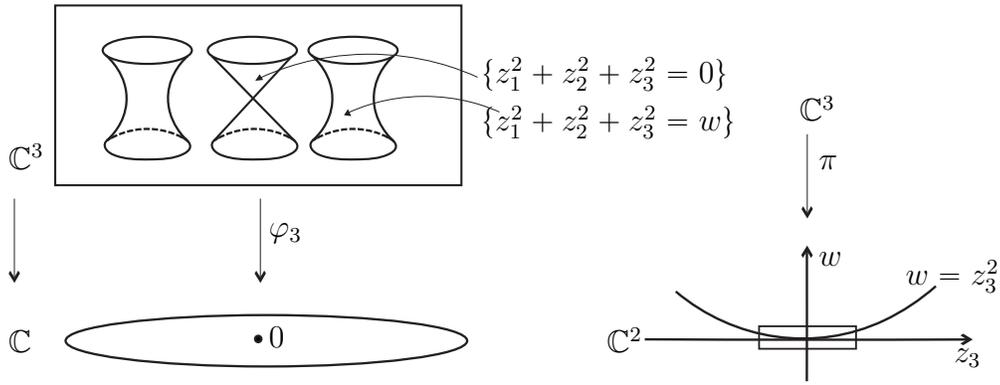
\begin{figure}[h]
\vspace{2mm}
\begin{overpic}[scale=.9]{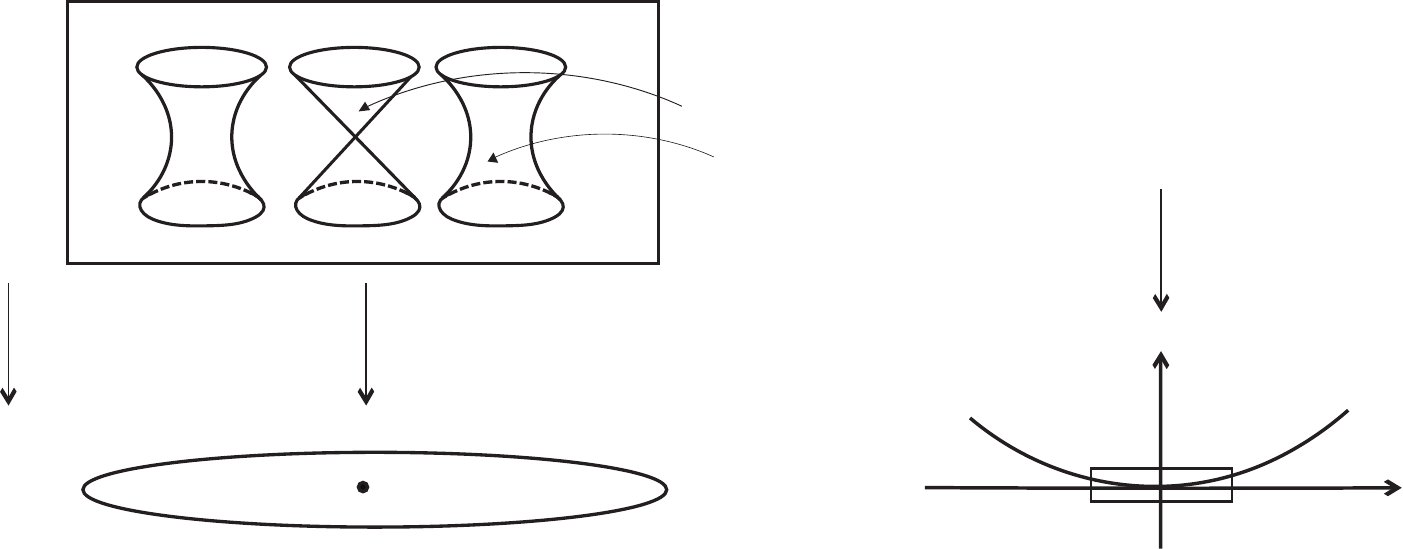}
    \put(49,31){$\{z_1 ^2+z_2 ^2+z_3 ^2=0\}$}    
	\put(49,26){$\{z_1 ^2+z_2 ^2+z_3 ^2=w\}$}    
	\put(82,27){$\C^3$}    	
	\put(0,3){$\C$}
	\put(62,3){$\C^2$}
	\put(27,3.5){$0$}
	\put(0,22){$\C^3$}
	\put(84,22){$\pi$}
		\put(27,15){$\varphi_3$}	
		\put(98,2){$z_3$}
		\put(84,12){$w$}
		\put(93,10){$w=z_3 ^2$}
    \vspace{5mm}
\end{overpic}
	\caption{The figure on the left is the projection map ${\varphi}_3$ onto $w$ whose regular fibers ${\varphi}_3^{-1}(w)$ are diffeomorphic to $T^*S^2$ and the figure on the right is the  projection map onto $z_3$ and $w$ whose regular fibers are $\{z_1 ^2+z_2 ^2=w-z_3 ^2\}$. Here ${\varphi_3} ^{-1}(0)$ is the singular fiber.}
    \label{fibers}
\end{figure}

Consider a sufficiently small box around the origin as depicted in Figure \ref{figure: projection pi}. Each vertical slice in that box will then represent the fiber $W^4$ of $f_3$. Recall that in order to show regularity, it is sufficient to prove the following local statement:
\begin{enumerate}
\item Each point on the vertical slice $W^4$ should be regular in that slice (See Figure \ref{figure: projection pi}). That is, all $J$-holomorphic curves in $W^4$ are transversely cut out.
\item There exists a biholomorphism such that $W^6{\cong} W^4 \times B_\delta$, where $B_\delta$ is a ball centered at 0 with radius $\delta$ in $\C$, for $\delta \geq 0$ small. 
\end{enumerate}

Second condition above implies that, in a sufficiently small box, the almost complex structure on $W^4$ which we will call $j^2\in \mbox{Aut}(TW^4)$ is invariant in the horizontal direction. That is to say, the almost complex structure on $W^6$ in some neighborhood of a critical point, which we will call $j^3\in \mbox{Aut}(TW^6)$, can be expressed as a split almost complex structure $(j^2, j_{\C})$, where $j_{\C}$ is the pullback of standard complex structure in $\C$, i.e., in the horizontal complex direction $z_3$.\sm

In light of the recall above, we need to write down an almost complex structure $j^3$ which works with the map $\mathbb{C}^3 \rightarrow \mathbb{C}$ such that 
\begin{enumerate}
\item if we take vertical slices in the box in Im($\pi$) as shown in Figure \ref{figure: projection pi}, then the slices will look like $T^*S^2$,
\item if we move in the horizontal direction $z_3$, then the almost complex structure is just a small variation of $j^3$ which is regular.
\end{enumerate}

Once we find such an almost complex structure $j^3$, we need to make sure that it extends to the rest of the total space. Firstly, observe that the standard complex structure on $\mathbb{C}^2$ induces a complex structure on $T^*S^2$. See Section \ref{inducedcpx} for details. Consider the projection map $\pi: \C^3 \to \C^2$ and a point $p\in \C^2$ close to the origin. We want to show that $\pi^{-1}(p)$ is transversely cut out with respect to the standard complex structure in $\C^3$. Observe that each vertical slice as shown in Figure \ref{figure: projection pi} is transversely cut out since the dimension of each vertical slice is $4$. In order to show that it is transversely cut out entirely, we need to show $j^2$ varies sufficiently slow in $z_3$-direction.\sm
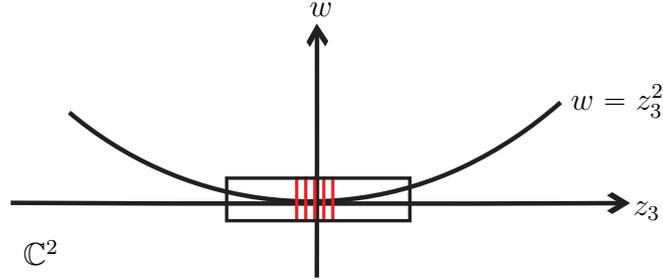
\begin{figure}[h] 
\begin{overpic}[scale=1.7]{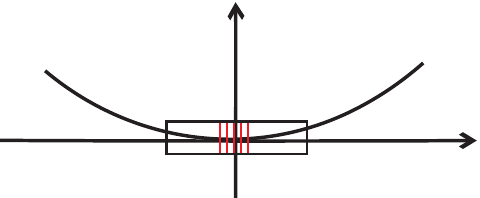}
    \put(90,27){$w=z_3 ^2$}
    \put(100,10){$z_3$}
	\put(48,42){$w$}    
	\put(2,2){$\C^2$}
    \vspace{5mm}
\end{overpic}
	\caption{The graph of the singular locus of the fibration $\pi$, where each red vertical line represents vertical slices in a small neighborhood of the origin.}
    \label{figure: projection pi}
\end{figure}

Let $\Theta: \C^4 \to \C^4$ be a biholomorphism defined as follows:
\begin{align*}
\Theta: (z_1, z_2, z_3, w) &\mapsto (z_1, z_2, z_3, w-z_3 ^2 = w'), \\
\{z_1 ^2 + z_2 ^2 + z_3 ^2=w\} &\mapsto \{z_1 ^2 + z_2 ^2=w' \}.
\end{align*}

The diffeomorphism $\Theta$ is equivalent to straightening out the singular locus where each vertical slice gets mapped to another vertical slice in the range as shown in Figure \ref{biholo}. Therefore, each vertical slice will be invariant in $z_3$-direction. Hence, the biholomorphism $\Theta$ makes the total space into a product which is sufficient to guarantee regularity.\sm

\begin{figure}[h]
\begin{overpic}[scale=.9]{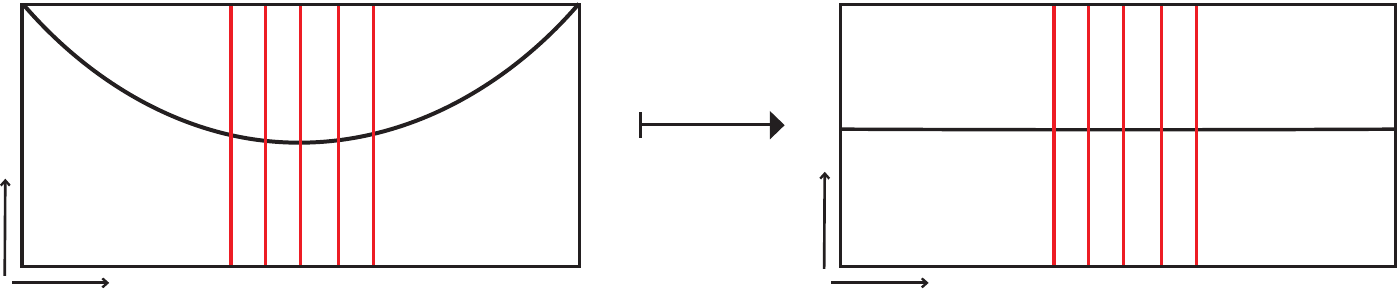}
	\put(50,13){$\Theta$}
	\put(-3,6){$w$}
  	\put(55,5,6){$w'$}
  	\put(6,-2){$z_3$}
  	\put(65,-2){$z_3$}  
  \end{overpic}
  \vspace{.5cm}
\caption{Biholomorphism $\Theta$ sending the singular locus $w=z_3 ^2$ to the horizontal line $w'=0$. Each box represents a small open neighborhood of singular locus around the origin. Red lines represent vertical slices.}
    \label{biholo}
\end{figure}
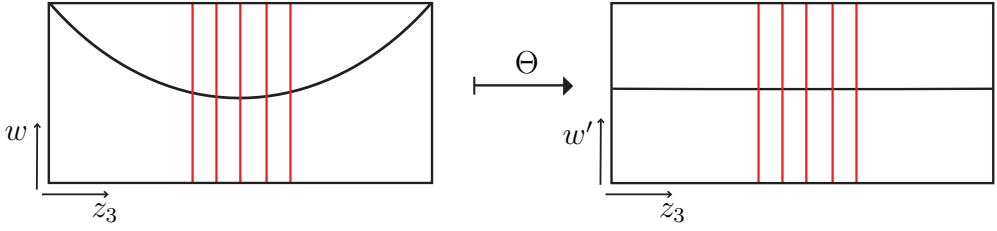

The existence of the biholomorphism $\Theta$ guarantees the fact that the induced complex structure $j^3$ on the total space is split in the neighborhood of a critical point and looks like symplectization at the ends. Therefore, all $J$-holomorphic curves are transversely cut out in a small neighborhood of a critical point. \sm

To extend the regularity in a neighborhood of a critical point to the rest of the neighborhood of the singular fiber which is a product manifold as explained in the first part of the proof, we assume that the generic cylindrical product almost complex structure $J^c$ we pick on the complement is regular. \sm

Let $j:= j^{n+1}= (j^{n}, j_{\C})$. By Theorem \ref{thm:gluing}, we can glue the regular almost complex structures $J^c$ and $j$ to get a regular almost complex structure on the neighborhood of the singular fiber. \sm

The construction above applies to all $n\geq3$ as the projection map $\pi$ onto the last two variables $z_n$ and $w$ will then correspond to the singular locus $w=z_n^ 2$ as in Figure \ref{fibers}, and thus, the induced almost complex structure will look like $j:= j^{n+1}= (j^{n}, j_{\D})$ for all $n\geq 2$. Thus, the regularity of the holomorphic curves follows. 

\end{proof}

\subsection{Gluing} \label{subsec:glu} 

The goal of this subsection is to state the gluing result that we need in Section \ref{ns} to examine the moduli space of planar $J$-holomorphic curves near a split holomorphic map in a $k$-level building. \sm

Consider the Lefschetz fibration $\hat{f}_n: \widehat{W}^{2n} \rightarrow \C$. Let $E_c$ be a neighborhood of a critical value $c \in \D$ of $f_n$ and $E$ be the corresponding neighborhood of the critical point above $E_c$.\sm

Recall that $j$ and $J^c$ are asymptotically cylindrical almost complex structures defined in a neighborhood of a critical point $E$ and $f_n^{-1}(E_c)- E$, respectively. Let $\SM_1$ denote the moduli space of asymptotically cylindrical planar $j$-holomorphic curves $u_1 : \dot{F}_1 \rightarrow \hat{E} \cong \ct n$. Similarly, denote by $\SM_2$ the moduli space of asymptotically cylindrical planar $J^c$-holomorphic curves $u_2 : \dot{F}_2 \rightarrow \hat{f}_n^{-1}(E_c)- E$.\sm

Let $\dot{F}=\dot{F}_1 \# \dot{F}_2$. Denote by $\mathcal{G}_{\delta}(\SM_1, \SM_2)$ the set of $J$-holomorphic maps $u: \dot{F} \to \hat{f}_n^{-1}(E_c)$ that are $\delta$-close to breaking into $u_1$ and $u_2$ in the sense of \cite{HT1}. \sm

Here, the almost complex structure $J$ on $u(\dot{F})$ is the gluing $J^c \# J_R \# j$, where $J_R$ is the almost complex structure on the neck region connecting $\hat{E}$ with $\hat{f}_n^{-1}(E_c)- E$. Then we state the following gluing result.

 \begin{thm} \label{thm:gluing}
 For sufficiently small $\delta >0$, there exist a sufficiently large $R$ and a gluing map
 \begin{align*}
 G: \SM_1 \times \SM_2 & \rightarrow \mathcal{G}_{\delta}(\SM_1, \SM_2)\\
 (u_1, u_2) & \mapsto u
 \end{align*}
which is a diffeomorphism onto its image. 
 \end{thm}
 
  Since the curves of interest are not multiply covered, the proof of this result is a consequence of the gluing results provided in \cite{Pa} and \cite{BH}.

\subsection{Neck stretching} \label{ns} In this subsection, we will study the degeneration of a regular fiber into a singular fiber via neck-stretching and show that the filling by planar $J$-holomorphic curves in the completion of $W^{2n}$ restricts to a foliation on $\C^n$ after neck-stretching on $E$. Hence, one can examine the structure of $J$-holomorphic curves by looking at $E$ while passing through the critical point. \sm

A priori, we have no control over $J$-holomorphic curves in the regular fiber $W^{2i}$ when they get pinched at the singular fiber. Therefore, one needs to analyze the degeneration of $J$-holomorphic curves across the singular fiber. Most reasonable procedure to study this is to use neck-stretching argument in SFT \cite{EGH}. Recall that neck-stretching is a deformation of an almost complex structure in some neighborhood of a contact hypersurface in $\widehat{W}$: See Section \ref{NS} and \cite{EGH} for a more detailed discussion. \sm

In what follows, we will apply neck-stretching to a neighborhood $E$ of a critical point in $W^{2n}$. In this way, $J$-holomorphic curves in this region will be forced to look like standard annuli when stretching is complete.  \sm

Consider the case $n=3$ and the map 
\begin{align*}
\varphi _3: \mathbb{C}^3 &\to \mathbb{C}, \\
(z_1, z_2, z_3) &\mapsto z_1^2+ z_{2}^2 +z_{3}^2
\end{align*} 
whose fibers are $\ct 2 \subset W^{4}$. Recall that the Liouville form $$\lambda=\dfrac{1}{2} \sum_{i=1}^2 (x_idy_i-y_idx_i)$$ on $T^*S^2$ induces a contact form on its boundary $ST^*S^2$. Recall also that
$$ST^*S^2\cong L(2,1)\cong S^3/\mathbb{Z}_2 \cong \mathbb{R}P^3.$$  

We take a small neighborhood $[-\epsilon, \epsilon] \times \R P^3$ of $\R P^3$. Following \cite{EGH} and \cite{HWZ2}, we can replace $[-\epsilon, \epsilon] \times \R P^3$ with $$[-t-\epsilon, t+\epsilon] \times \R P^3,$$ which we will call the neck region, since the almost complex structure on $[-\epsilon, \epsilon] \times \R P^3$ is chosen to be invariant under the Liouville flow. We also define the associated almost complex structure $J_t$ with the following properties:
\begin{enumerate}
\item It is invariant on $[-t-\epsilon, t+\epsilon] \times \R P^3$.
\item It agrees with the induced complex structure $j$ on $\ct 2$ elsewhere.
\end{enumerate}

Since $J_t$ is \textit{$(\RP^3, \lambda)$-compatible} on the neck region, one can send $t \to \infty$. As we take the limit, the almost complex structure $J_t$ will start deforming so that in the limit it breaks $W^4$ into the following two noncompact manifolds with cylindrical ends:
\begin{itemize}
\item $\D^*S^2 \cup [0, \infty) \times \R P^3 \cong \ct 2$ 
\item $\widebar{W}^4=W^4$ $\setminus$ $\ct 2$
\end{itemize}

One can then use SFT compactness \cite{BEHWZ} to observe that the $J$-holomorphic curves in $\widehat{W}^4$ converge to a holomorphic building. \sm

One can then conclude that $\widebar{W}^4$ is foliated by planar $J$-holomorphic curves away from a set of binding orbits of $\partial W^4$. Hence, filling by planar $J$-holomorphic curves in the completion of $W^{2n}$ restricts to the foliation on $T^*S^n$ after neck-stretching in a neighborhood of a singular fiber. \sm

Recall that under the identification of $\ct n$ with the tangent bundle $TS^n$ determined by a fixed a Riemannian metric $g$ on $S^n$, one can observe that the Reeb vector field on $ST^*S^n$ corresponds to the geodesic flow on $S^n$.  Thus, the Reeb orbits are closed geodesics of $S^n$. \sm

Observe that this is a highly degenerate situation since these closed orbits, i.e. geodesics, are not isolated, but come in families. That is, there exists a Morse-Bott family of closed Reeb orbits. This degenerate situation can be converted into a nondegenerate situation by perturbing our contact form in such a way that it induces the same contact structure on the unit sphere bundle $ST^*S^n$ whose Reeb flow generates the desired nondegenerate orbits. \sm

Let us begin our analysis with the case $n=2$. Recall that the closed Reeb orbits of $ST^*S^2$ are closed geodesics of $S^2$. We wish to show that there exists a contact form on $ST^*S^2\cong \mathbb{R}P^3$ whose associated Reeb vector field has precisely two nondegenerate closed orbits of Conley-Zehnder index 1 to obtain a moduli space of planar pseudoholomorphic curves with (Fredholm) index 2. \sm

We consider the class of $j$-holomorphic curves in $\ct 2$ that are asymptotic to 2 closed nondegenerate Reeb orbits. Call these two orbits $\gamma^+_1$ and $\gamma^+_2$ analogously defined as $\gamma$ in Section \ref{section:setup}. Let $u$ be a planar $j$-holomorphic curve for the Lefschetz fibration $\ct 2 \to \D$ with fibers $\ct 1$. Then
$$\operatorname{ind}(u)=(n-3)\chi(\dot{F})+\mu_{\text{CZ}}(\gamma^+_1, \tau)+\mu_{\text{CZ}}(\gamma^+_2, \tau).$$

Observe that $\chi(\dot{F})=2-2g-k=2-0-2=0$. In order for these curves to be of index 2, the following should hold:
$$\mu_{\text{CZ}}(\gamma^+_1, \tau)=\mu_{\text{CZ}}(\gamma^+_2, \tau)=1.$$
with respect to a properly chosen trivialization $\tau$.  Therefore, we need to perturb our contact form so that the resulting closed Reeb orbits are of index 1. This perturbation is crucial for determining the following: when we stretch the neck, we have a priori no information about what collection of orbits the curve in $T^* S^2$ limits to. To remedy the situation, we will use the perturbation of the contact form described in \cite[Section 2]{W2} which has been constructed by using the abstract open book of the associated planar contact manifold. Therefore, $\ct 2$ region is foliated by cylindrical $j$-holomorphic curves asymptotic to closed Reeb orbits of \textit{Conley-Zehnder index 1}. \sm

Note that one can extend this procedure inductively to higher dimensions by using the iterative argument on $\ct n$. Finally, we need to glue back the cylindrical ends of each level in the holomorphic building to study the foliation of $\widehat{W}^4$ by using Theorem \ref{thm:gluing}. \sm

Note also that one can apply this construction to further generalize it to $\widehat{W}^{2n}$, for $n \geq 2$, inductively by using the facts that each $W^{2i}$ is a fiber of the Lefschetz fibration on $W^{2i+2}$ for all $i\geq 2$ and the evaluation map for $\widehat{W}^{2i}$ is of degree 1 for each $i > 2$.\sm

\subsection{Proof of Proposition \ref{main_prop}} In this section, we will use the results of the previous sections and prove Proposition \ref{main_prop} by arguing that the evaluation map 
\begin{align*}
\operatorname{ev}: \mathcal{M}_{\widehat{W}^{2i}}(J) &\to \widehat{W}^{2i}, \\
(u, z) &\mapsto u(z)
\end{align*}
 is of degree 1 for $i \geq 2$, i.e., through any generic point $p \in \widehat{W}^{2i}$, the mod 2 count of $\operatorname{ev}^{-1}(p)$ is 1. Consider $W^{2i}=f_i^{-1}(p)$. When $i=2$, one can make a stronger statement, i.e., the completion $\widehat{W}^4$ admits a foliation by embedded finite-energy planar $J$-holomorphic curves \cite{W2}.\sm

Now assume that $\widehat{W}^{2i}$ is filled by planar $J$-holomorphic curves and show that $\widehat{W}^{2i+2}$ is also filled by planar $J$-holomorphic curves. Since $f_{i+1}$ is an iterated planar Lefschetz fibration, we can conclude that each regular fiber of $\hat{f}_{i+1}$ is filled by planar $J$-holomorphic curves for $i\geq2$.\sm

In Sections \ref{subsec:regularity}, \ref{subsec:compactify}, and \ref{indexformula}, we showed that near a critical point, the moduli space of planar $J$-holomorphic curves in $\widehat{W}_s^{2i+2}$ is still regular, compact and has the expected dimension $2i$ when ${\mathcal{M}}_{\widehat{W}_s^{2i}}(J)$ is regular, compact and has the expected dimension. Then the evaluation map on ${\mathcal{M}}_{\widehat{W}_s^{2i+2}}(J)$ is still a degree 1 map and therefore $\widehat{W}^{2i+2}$ is filled by planar $J$-holomorphic curves .\sm

In light of the discussion above, we can use the evaluation map

\begin{align*}
\mbox{\text{ev}}: {\mathcal{M}}_{{\widehat{W}}^{2n}}(J) &\longrightarrow \widehat{W}^{2n},\\
(u, z)&\longmapsto u(z),
\end{align*}
where $z$ is a marked point in $W^{2n}$, to show that there exists a unique curve passing through each generic point in $\widehat{W}^{2n}$ up to algebraic count since the completion of each regular and singular fiber of $f_n: W^{2n} \to \D$ is filled by planar $J$-holomorphic curves. Hence, the pages $W^{2n}$ of the open book supported by the contact manifold $M$ are filled by planar $J$-holomorphic curves. This finishes the proof of Proposition \ref{main_prop}.

\qed

\subsection{Proof of the Main Theorem} In this section, we use notions such as $(M, \lambda)$-compatible and $(M, \lambda)$-neck-stretching. For their definitions, we refer the reader to Definition \ref{defn:Mcomp} and \ref{defn:NSofM}, respectively.\sm

\begin{lemma} There exists an almost complex structure $J_0$ on $\R \times M$ compatible with a suitable stable Hamiltonian structure $\mathcal{H}_0$ as in Lemma \ref{lemma:wendl} such that the $J$-holomorphic curves in the pages $W^{2n}$ of the open book on $M$ lift to $J_0$-holomorphic curves in $\R \times M$ with index $2n$.
\end{lemma}

\begin{proof}
Pick an almost complex structure $J$ on the pages $\widehat{W}^{2n}$ of the open book supported by $M$ given by Proposition \ref{prop:regular}. Let $h: W^{2n} \to W^{2n}$ be the monodromy map of the open book on $M$ and $$M_h=[0, 1] \times W^{2n}  \bigm/ (0, h(z)) \sim (1, z)$$ be the associated mapping torus which is a smooth $(2n+1)$-dimensional manifold with boundary. Then there exists a natural fibration $$\phi: M_h \to S^1$$ whose fiber is $W^{2n}$. When we move along the base $S^1$ as depicted in Figure \ref{interpol}, the almost complex structure $J$ on each fiber starts changing due to the monodromy map $h$. By the time we come to the point $q\in S^1$, the almost complex structure $h_{*}(J)$ on the fiber over $q \in S^1$ is not the same as the almost complex structure $J$ over $p \in S^1$. 

\begin{figure}[h]
\begin{overpic}[scale=.7]{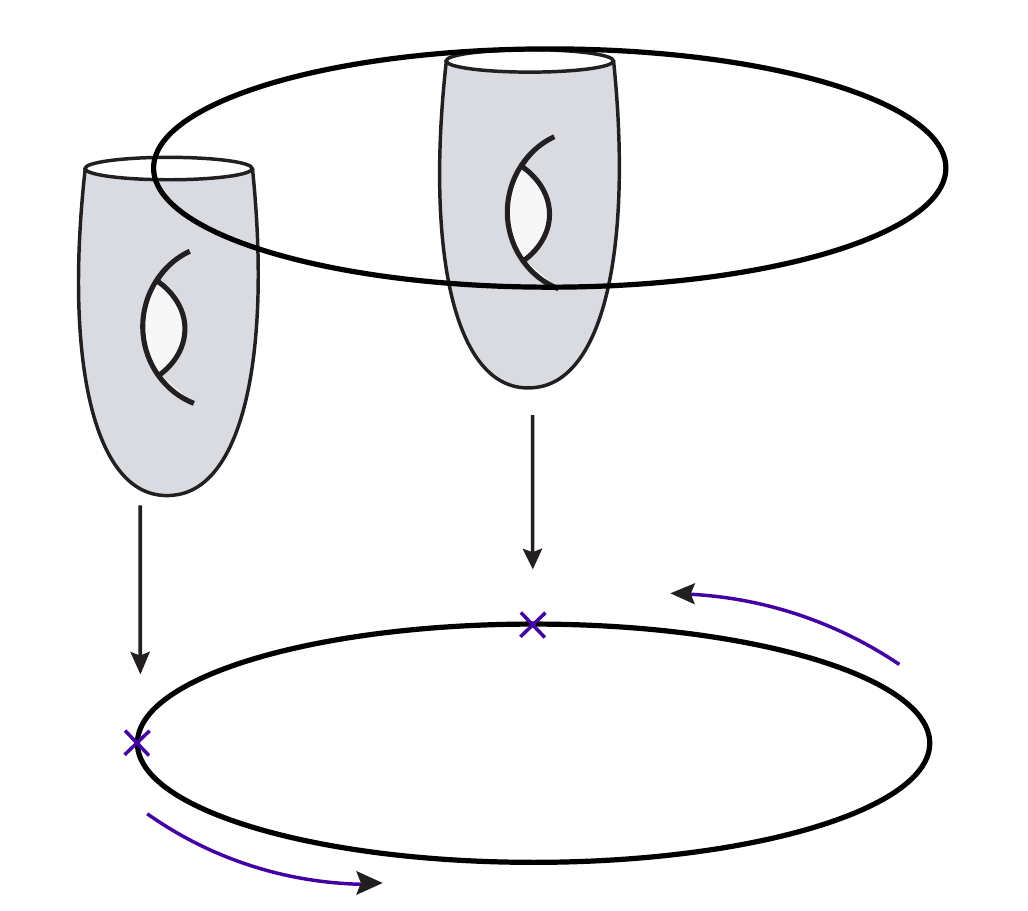}
\put(95, 15){$S^1$}
\put(95, 71){$M_h$}
\put(9, 16){$p$}
\put(51, 23){$q$}
\put(20, 1){$h$}
\put(5, 45){$J$}
\put(32.5, 52){$h_{*}(J)$}
\end{overpic}
\caption{The figure representing the action of the monodromy on the fibers $W^{2n}$ and the associated almost complex structures. Here $p$ and $q$ represent points on $S^1$. The induced map $h_{*}:TW^{2n} \to TW^{2n}$ is identity on a collar neighborhood $\mathcal{N}(\partial M)$ of $\partial M$ while fixing $J=h_{*}(J)$ on $\mathcal{N}(\partial M)$.}
\label{interpol}
\end{figure}
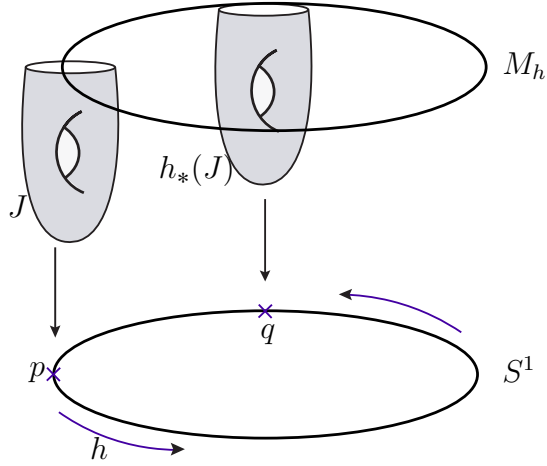

Let $\sigma: [0, 1] \to S^1$ be a path with $\sigma(0)=q$ and $\sigma(1)=p$. Since there exists a curve through each point over the path $\sigma$ on $M_h$ by Proposition \ref{main_prop}, we can interpolate the almost complex structure over $\sigma$. Note that we cannot explicitly construct this interpolation. However, one can use the evaluation map over this path to show that, geometrically through each point, there exists a curve. Hence, the interpolation follows. \sm

Now we need to check if the almost complex structure $J$ on the mapping torus $M_h$ extends to the symplectization $\R\times (B \times \D)$ of a neighborhood of the binding $B=\partial(W^{2n})$. 

To do this, we need to find a foliation of $\R\times (B \times \D)$ by $J$-holomorphic cylinders asymptotic to the Reeb orbit $\gamma^{+}$ (See Section \ref{section:setup} for the description of the Reeb orbit $\gamma^{+}$) that fit smoothly with the $J$-holomorphic curves in the fiber $\{\text{constant}\} \times W^{2n}$. To see this, we refer the reader to \cite[Page 7]{W2} for a detailed analysis near the binding. One can then attach the symplectization of $$\underbrace{\{0\} \times \{0\} \times \dots \times \{0\}}_{{(n-1)}\text{-many}} \times \gamma^+$$ to the holomorphic fiber $\{\text{constant}\} \times W^{2n} \subset \R \times M_h$ which results in an extension of a $J$-holomorphic fiber in $\R \times M_h$ to a $J_0$-holomorphic fiber in $\R \times M$. \sm

Hence, given a stable Hamiltonian structure $\mathcal{H}_0$ as in Lemma \ref{lemma:wendl}, $J_0$ is an almost complex structure on $\R \times M$ compatible with $\mathcal{H}_0$ such that the $J$-holomorphic curves in the pages $W^{2n}$ of the open book on $M$ lift to $J_0$-holomorphic curves in $\R\times M$. This completes the construction of $J_0$-holomorphic curves on $\R \times M$ except for the index calculation.\sm

Next we compute this index. Let $u \in \mathcal{M}_{\R \times M}(J_0)$ be a planar $J_0$-holomorphic curve in $\R \times M$. Recall that the expected dimension of the moduli space of $J_0$-holomorphic curves in $\R \times M$ is the Fredholm index of $\bar{\partial}_{J_0} u$. This index is

\begin{equation}\label{indformula}
\mbox{ind}(u)=((n+1)-3)\chi (\Sigma)+\sum \limits_{\gamma^+} \mu_{\text{CZ}}(\gamma ^+, \tau)-\sum \limits_{\gamma^-}\mu_{\text{CZ}}(\gamma^-, \tau)+2c_1(u^{*}T(\R \times M))
\end{equation} 
where $\gamma^+$ and $\gamma^-$ are positive and negative ends of $u({\dot{F}})$, respectively and $\tau$ is the trivialization of $\xi$ along $\gamma^{\pm}$. Note that the Conley-Zehnder index is independent of the choice of $\tau$. The dimension of the ambient manifold is $(2n+2)$, which explains the term $((n+1)-3)$.\sm

Let $k$ be the number of positive ends of $u(\dot{F})$. Recall that we do not have any negative ends. Therefore, $\mu_{\text{CZ}}(\gamma^-, \tau)=0$. Let us begin with the case $n=2$. With respect to the trivialization $\tau$ for which boundary of the pages are longitudinal, the Conley-Zehnder index of each of the Reeb orbits is $\mu_{\text{CZ}}(\gamma^+, \tau)=1$ for the elliptic case. \sm

The relative first Chern number of the bundle $u^{*}T(\R \times M) \to \dot{F}$ with respect to the trivialization $\tau$ is
$$
c_1(u^{*}T(\R \times M))=c_1(T\dot{F} \oplus N_u)=c_1(T\dot{F})+c_1(N_u)
$$
where $N_u$ denotes the normal bundle to $u$. Observe also that transverse direction is trivial. Thus, we have $$c_1(u^{*}T(\R \times M))=c_1(T\dot{F}).$$

It remains to analyze the tangential component of $c_1$. Note that the relative Chern class of a planar surface $\dot{F}$ with $k$ positive ends is given by
\begin{equation*}
<c_1(T\dot{F}, \tau), \dot{F}>=<e(\dot{F}), \dot{F} >=\chi(\dot{F})=2-k.
\end{equation*}

Having determined the relative first Chern number, we shall compute the index for $n=2$: 
\begin{align*}
\text{ind}(u)&=(2-3)(2-0-k)+k-0+2(c_1(T\dot{F}))\\
&=-1(2-k)+k+2(2-k)=2.
\end{align*}

Now consider the case when $u: (\dot{F}, j) \longrightarrow (\R \times M, J_0)$ and dim$_{\R}(\R \times M)=2n+2$. Consider also a projection map from a neighborhood $\mathcal{N}(M^{2n-3})$ of $M^{2n-3}$ to $\D$. This map gives a trivialization $\tau$ of $\mathcal{N}(M^{2n-3}) \cong M^{2n-3} \times \D$. We then obtain $$\mu_{\operatorname{CZ}}(\gamma^+, \tau)=(n+1)-1=n.$$ Thus, Equation \ref{indformula} implies that  

\begin{align*}
\text{ind}(u)&=((n+1)-3)\chi(\dot{F})+k\cdot n-0+2\chi(\dot{F})\\
&=(n-2)(2-k)+k\cdot n+2(2-k)\\
&=n(2-k)+k\cdot n\\
&=2n.
\end{align*}
\end{proof}

Let $\lambda_0$ be a nondegenerate contact form on an iterated planar contact manifold $M$ and let $$\lambda=f \cdot \lambda_0,$$ where $f$ is a positive function on $M$ and $\lambda$ is a nondegenerate contact form on $M$. \sm

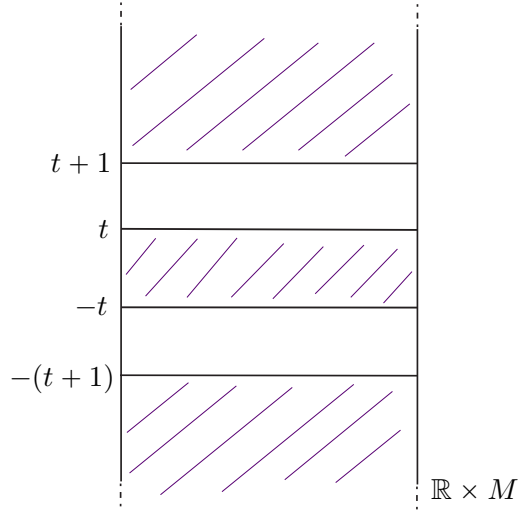
\begin{figure}[h]
\vspace{-.5cm}
\begin{overpic}[scale=.7]{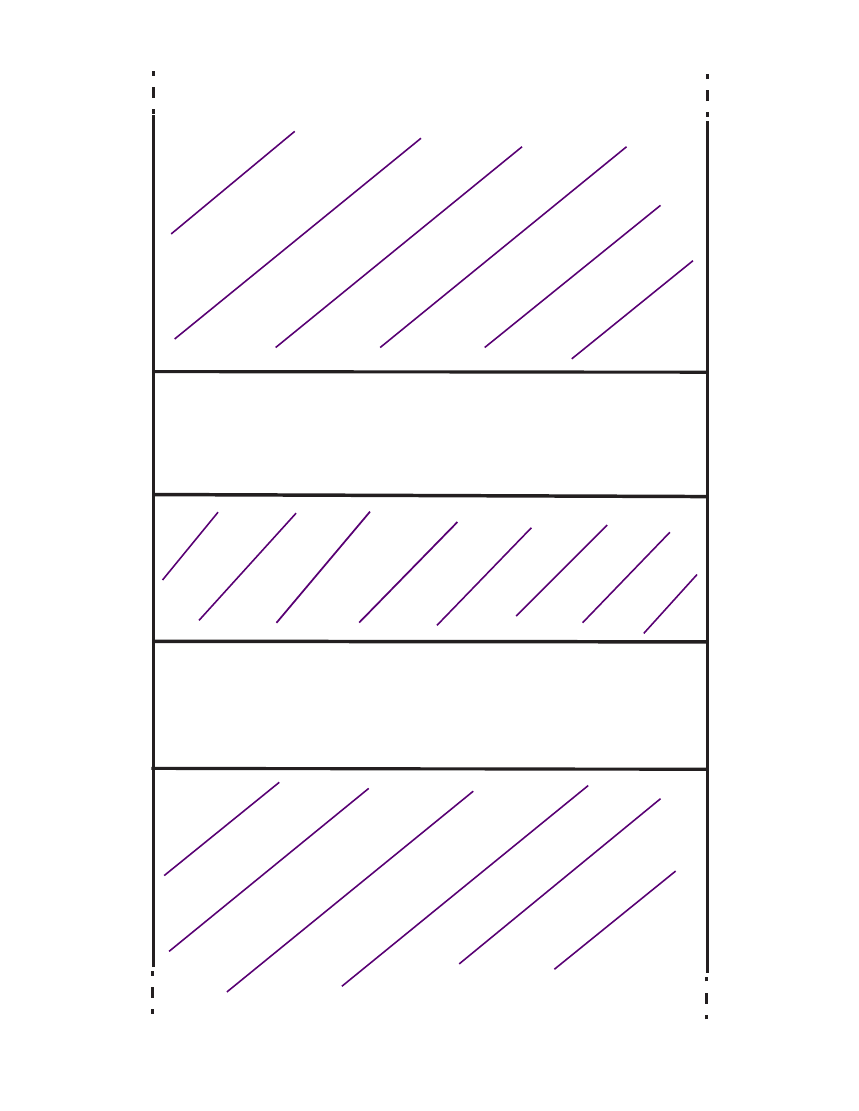}  
\small{
\put(2, 65){$t+1$}
\put(10, 54){$t$}
\put(6,41){$-t$}
\put(-5, 29){$-(t+1)$}}
\put(66, 10){$\R\times M$}
  \end{overpic}
\vspace{-.5cm}
\caption{The degeneration of the almost complex structure $J_0$ on $\R \times M$.}
    \label{mainthm}
\end{figure}

Let $I_t=[-t, t]$ and $t \in [1/2, \infty)$. Denote by $\tilde{J}_t$ an $(M, \lambda)$-compatible almost complex structure on $I_t \times M$. The almost complex structure $J_t$, for $t \geq \frac{1}{2}$, is defined as follows:
\[   
J_t = 
     \begin{cases}
       \tilde{J}_t &\text{on } [-t, t] \times M,\\
        J_0&\text{on } [t+1, \infty) \times M, \\
       J_0 &\text{on } (-\infty, -(t+1)] \times M.\\
     \end{cases}
\]

Assuming that $\tilde{J}_{t}$ is invariant in the Liouville direction, the almost complex structure $J_t$ on $\R \times M$ will then be the gluing of the almost complex structure $\tilde{J}_{t}$ on $I_t \times M$ with $J_0$ on the complement of $I_t \times M$ in $\R \times M$.\sm

The family $\{J_t\}_{t=1/2}^\infty$ is the $(M, \lambda)$-neck-stretching of $J_{1/2}$. By Proposition \ref{main_prop} and Lemma \ref{lemma}, for each $t \in [0, \infty)$, there exists a planar $J$-holomorphic curve $u_t$ which passes through $\{0\} \times M$.\sm

Let $z_t$ be a cylinder in the domain of a $J_t$-holomorphic curve $u_t$ in $I_t \times M$. Let $a(t)<t$. After possible $\R$-translation, the image $u_t(z_t)$ is close to $[0, a(t)] \times \gamma^{+}$ as $a(t) \to \infty$. Now, by Gromov-Hofer compactness, after passing to a subsequence of $u_t$, planar holomorphic curves in $I_t \times M$ will be asymptotic to closed Reeb orbits as $t \to \infty$. This concludes the proof of Main Theorem.

\qed

Notice that as $t \to \infty$, we need to make sure that compactness issues such as bubbling or breaking of curves do not arise. In what follows, we show that these degenerations never happen in our setting.

\subsubsection{Compactness of the moduli space} \label{subsec:compactify}

Here we study the compactness of the moduli space $\mathcal{M}_{\R \times M}(J_0)$ by using the relevant compactness results in the symplectic field theory and also by analyzing the existence of broken cylinders, and bubbling. For a more detailed exposition on the subject, we refer the reader to \cite{BEHWZ}. \sm

The compactness result in the symplectic field theory \cite{BEHWZ} states that holomorphic buildings obtained via neck-stretching consist of finite energy curves that are asymptotic to holomorphic cylinders on closed Reeb orbits. Moreover, asymptotic Reeb orbits must match up in pairs. That is, each positive end of a curve in level $i$ is matched with the corresponding negative end of the curve in level $i+1$. \sm

During the neck-stretching process described in Section \ref{NS}, one needs to keep track of possible degenerations of $J$-holomorphic curves into bubbles or broken cylinders. Next, we will study these degenerations and prove the following: 

\begin{lemma} \label{lemma} There exist no broken cylinders and bubblings of closed curves in $\mathcal{M}_{\R \times M}(J_0)$.
\end{lemma} 

\begin{proof}
By Stokes' theorem, there exist no non-constant $J$-holomorphic spheres since Weinstein domains are exact manifolds. Therefore, bubbling of closed curves never occurs in all dimensions.\sm

As for breaking, one needs to perform a careful analysis. Recall that the monodromy map $h: W^{2n} \to W^{2n}$ of the open book on $M$ and the associated mapping torus $$M_h=[0, 1] \times W^{2n}  \bigm/ (0, h(z)) \sim (1, z).$$ Then there exists a natural fibration $$\phi: M_h \to S^1$$ whose fiber is $W^{2n}$. When we move along the base $S^1$ as depicted in Figure \ref{interpol}, the almost complex structure $J$ on each fiber starts changing due to the monodromy $h$. Recall that we interpolated the almost complex structures above the path $\sigma \in S^1$ to ensure that the almost complex structures are homotopic.\sm

Next we will examine possible degenerations of curves, such as breaking as we are changing the almost complex structure $J$ during this interpolation. To see this, consider the map $$\hat{\phi}: \R \times M \to \C.$$ 

Notice that breaking may occur if we obtain a holomorphic building with 2 levels, none of which is trivial, after symplectization. \sm

Let $n=2$. Consider the map $\hat{\phi}: \R \times M \to \C$ where $\dim_{\R} M=5$. Observe that the cylindrical ends in the fiber direction containing the Reeb orbits map down to a point. Hence, by the open mapping theorem, every end containing Reeb orbits has to live in a fiber since $\hat{\phi}$ is a holomorphic map.\sm

\begin{figure}[h]
\begin{overpic}[scale=1.2]{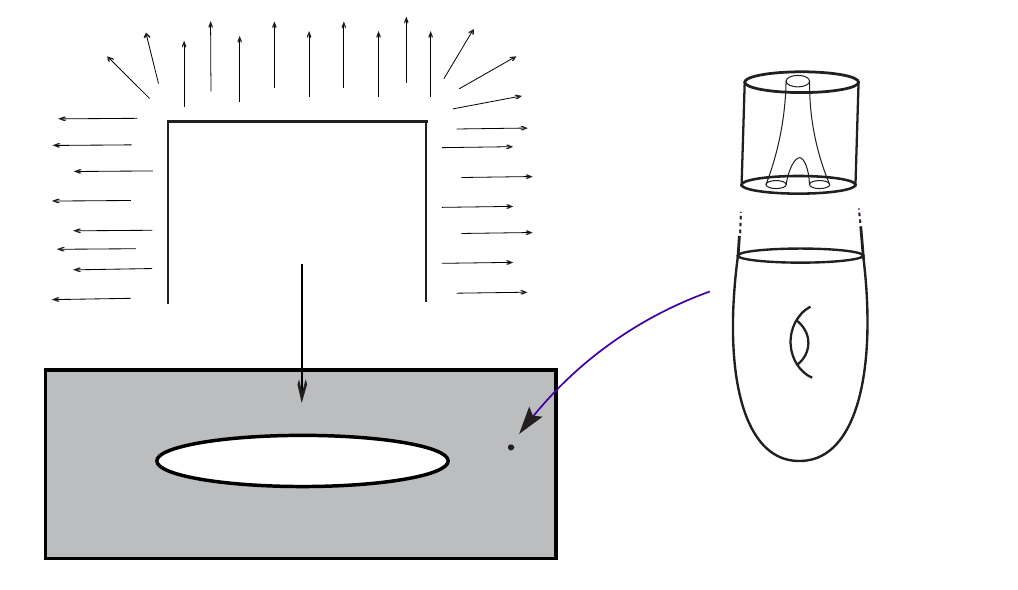}
\put(51, 5){$\C$}
\put(48, 13.5){$z$}
\put(86, 17){$\hat{\phi}^{-1}(z)$}
\put(77.5, 44){$u$}
\put(86, 45){$\R \times \partial W^{2n}$}
\put(20.5, 11.5){$S^1$}
\put(3, 48){$\R \times M$}
\put(26,26){$\hat{\phi}$}
\end{overpic}
\caption{The figure on the left is the map $\hat{\phi}: \R \times M \to \C$. The figure on the right is the fiber $\hat{\phi}^{-1}(z)$ over $z \in \C$ and a 2-level building depicting the breaking of curves. Here $u$ represents the end of a broken curve containing the Reeb orbit in $\widehat{W}^{2n}$.}
\label{breaking}
\end{figure}

Let $u$ be a curve in the cylindrical end $\R \times \partial W^4$ containing the Reeb orbit depicted as in Figure \ref{breaking}. Hence, by the discussion above, $u$ has image inside $W^4$ and $W^4$ admits a planar Lefschetz fibration whose fibers are planar surfaces with no negative ends. Therefore, no such $u$ exists unless it is a trivial cylinder and trivial cylinders are ignored in holomorphic buildings by convention.\sm

Observe that this construction applies to $W^{2n}$ for all $n \geq 2$ inductively since the map $\hat{\phi}$ is a projection map from the completion of the page $W^{2n}$ to a point for all $n$, Thus, we could go one dimension down to use the information that $u$ is contained in a fiber. Therefore, the proof follows.\sm
\end{proof}

At this point, it is crucial to remind the reader that the almost complex structure $J$ is not cylindrical as explained in Section \ref{inducedcpx}. It is rather asymptotically cylindrical and this may affect the compactness discussion above. However, thanks to \cite{B}, we know that the compactification results in SFT generalizes to asymptotically cylindrical almost complex structures. 

\bibliographystyle{amsalpha}

\end{document}